\documentclass{amsart}
\usepackage[utf8]{inputenc}

\usepackage{hyperref}
\usepackage{enumitem}
\usepackage{tikz-cd}
\tikzset{
	symbol/.style={
		draw=none,
		every to/.append style={
			edge node={node [sloped, allow upside down, auto=false]{$#1$}}}
	}
}
\usepackage{mathtools,amsmath,amssymb}
\usepackage{bbm}

\usepackage{tikz}
\usetikzlibrary{cd,arrows}
\tikzcdset{scale cd/.style={every label/.append style={scale=#1},
		cells={nodes={scale=#1}}}}
\usetikzlibrary{decorations.markings}

\usepackage[backend=biber, style=alphabetic, firstinits=true]{biblatex} 
\theoremstyle{plain}
\newtheorem{theorem}{Theorem}[section]
\newtheorem{lemma}[theorem]{Lemma}

\newtheorem{cor}[theorem]{Corollary}
\newtheorem*{ack}{Acknowledgements}
\theoremstyle{definition}
\newtheorem{defi}[theorem]{Definition}
\newtheorem{ass}[theorem]{Assumption}
\newtheorem{remark}[theorem]{Remark}
\newtheorem{exam}[theorem]{Example}
\newtheorem{nota}[theorem]{Notation}

\newcommand{\e}{\varepsilon}

\newcommand{\Oe}{{\Omega_\e}}
\newcommand{\Oes}{{\Omega_\e^\mathrm{s}}}
\newcommand{\Ge}{{\Gamma_\e}}
\newcommand{\Yp}{{Y^\mathrm{p}}}
\newcommand{\Ys}{{Y^\mathrm{s}}}
\newcommand{\Ypx}{{Y^\mathrm{p}_x}}
\renewcommand{\div}{\operatorname{div}}
\newcommand{\R}{\mathbb{R}}
\newcommand{\Q}{\mathbb{Q}}
\newcommand{\Z}{\mathbb{Z}}
\newcommand{\N}{\mathbb{N}}
\newcommand{\xy}{(x,y)}
\newcommand{\tx}{(t,x)}
\newcommand{\txy}{(t,x,y)}
\newcommand{\norm}[2]{\left|\left|#1\right| \right|_{#2}}

\newcommand{\HepsN}{{H^1_{\Ge}(\Oe)^N}}
\newcommand{\Leps}{{L^2(\Oe)}}
\newcommand{\HYperN}{{L^2(\Omega;H^1_{\Gamma \#}(\Yp))}}
\newcommand{\intOe}{\int\limits_{\Oe}}
\newcommand{\intO}{\int\limits_{\Omega}}
\newcommand{\intYp}{\int\limits_{\Yp}}
\newcommand{\intOY}{\intO}
\newcommand{\intY}{\int\limits_{Y}}
\newcommand{\intOYp}{\intO \intYp}
\newcommand{\xeps}{\frac{x}{\e}}
\newcommand{\xxeps}{\left(x,\xeps\right)}

\newcommand{\1}{\mathbbm{1}}

\newcommand{\dx}{dx}
\newcommand{\dy}{dy}
\newcommand{\dz}{dz}
\newcommand{\dyx}{\dy\dx}

\newcommand{\Te}{\mathcal{T}_\e}
\newcommand{\xeY}{\e\left[\frac{x}{\e}\right]_Y}
\newcommand{\limEpsNull}{\lim\limits_{\varepsilon \to 0}}

\title[Homogenisation of the Stokes eq.~for evolving microstructure]{Homogenisation of the Stokes equations for evolving microstructure}
\author{David Wiedemann}
\author{Malte A. Peter}
\thanks{D.W. is partially supported by a doctoral scholarship of Studienstiftung des deutschen Volkes}
\address{(D.W.) Institute of Mathematics, University of Augsburg, 86135 Augsburg, Germany}
\email{david.wiedemann@math.uni-augsburg.de}
\address{(M.A.P.) Institute of Mathematics, University of Augsburg, 86135 Augsburg, Germany and Augsburg Centre for Innovative Technologies, Universität Augsburg, 86135 Augsburg, Germany}
\email{malte.peter@math.uni-augsburg.de}


\begin{document}
\begin{abstract}
We consider the homogenisation of the Stokes equations in a porous medium which is evolving in time. At the interface of the pore space and the solid part, we prescribe an inhomogeneous Dirichlet boundary condition, which enables to model a no-slip boundary condition at the evolving boundary. We pass rigorously to the homogenisation limit with the two-scale transformation method. In order to derive uniform a priori estimates, we show a Korn-type inequality for the two-scale transformation method and construct a family of $\e$-scaled operators $\div_\e^{-1}$, which are right-inverse to the corresponding divergences.
The homogenisation result is a new version of Darcy's law. It features a time- and space-dependent permeability tensor, which accounts for the local pore structure, and a macroscopic compressibility condition, which induces a new source term for the pressure. In the case of a no-slip boundary condition at the interface, this source term relates to the change of the local pore volume.
\end{abstract}
\maketitle
\tableofcontents
\section{Introduction}
Effective fluid flow in completely saturated porous media can be described by Darcy's law. Based on results of experiments, it was formulated by Henry Darcy in \cite{Dar56}. This empirical law can be justified by means of homogenisation theory on the basis of general laws of fluid dynamics. Using the formal two-scale expansion, the Darcy's law has been derived from the Stokes equations quite early (see for example \cite{Kel80},\cite{Lio81},\cite{San80}). Assuming a periodic porous medium with disconnected solid matrix, L.~Tartar proved rigorously the convergence of the homogenisation for the following $\e$-scaled Problem in \cite{Tar79}.
He considered the Stokes equations in a periodic porous medium $\Omega$ with fluid phase $\Oe$ (of period $\e$) with homogeneous Dirichlet boundary conditions
\begin{align}\label{eq:StrongStokesConstantDomain}
- \e^2 \nu \Delta v_\e + \nabla p_\e = f  \textrm{ in } \Oe , &&
\div(v_\e)= 0 \textrm{ in } \Oe , &&
v_\e= 0 \textrm{ in } \partial\Oe, 
\end{align} 
where $v_\e$ and $p_\e$ denote the velocity and pressure of the fluid and $f$ the density of the forces acting on the fluid.
He proved that the extension of $v_\e$ to $\Omega$ by $0$ converges weakly to $v$ in $L^2(\Omega)$ and $P_\e$, which is an extension of the pressure $p_\e$  too the solid part of $\Omega$, converges strongly in $L^2(\Omega) / \R$ to $p$, where $v$ and $p$ are the unique solutions of
\begin{align}\label{eq:DarcysLaw}
v = \frac{1}{\nu}K(f- \nabla p) \textrm{ in } \Omega, &&
\div(v) = 0  \textrm{ in } \Omega, &&
v \cdot n= 0 \textrm{ in } \partial \Omega,
\end{align}
and $K$ is a positive definite symmetric permeability tensor which is given by solutions of Stokes problems on a reference cell. 
The main task in the proof of the convergence is the derivation of an $\e$-independent bound for $p_\e$. At this point, Tartar had to assume that the solid part of a cell is strictly contained inside the cell.
Extending the ideas of L.~Tartar, G.~Allaire could omit this assumption and proved the convergence for more general domains in ~\cite{All89}.
In \cite{All91} G.~Allaire gave some corrector results, which yield a strong convergence of the velocity in $L^2(\Omega)^N$.
	
The purpose of the present paper is to extend the problem \eqref{eq:StrongStokesConstantDomain} to a domain $\Oe(t)$ whose microstructure is evolving in time and to prove the convergence of the homogenisation for this new setting.
The case of an evolving microstructure is motivated by many different physical, chemical and biological applications such as, crystal precipitation and dissolution processes or biofilm growth models (e.g.~see~\cite{Pet07a}, \cite{PB09}, \cite{Noo08}, \cite{REK15}, \cite{RNF12}, \cite{SK17}, \cite{SK17a}). 
Some of these publications consider models which include the Stokes equations on an evolving domain. Using the method of formal two-scale expansion there are some effective models derived, which include also a Darcy law (cf.~\cite{SK17}).
	
In this present paper, we consider the rigorous homogenisation of the Stokes equations for $\e$-scaled domains $\Oe(t)$ which are evolving in time.
Thus, we consider the Stokes problem
\begin{align}\label{eq:StrongStokesEvolutionaryDomain1}
- \e^2 \nu \div( 2e(v_\e)) + \nabla p_\e &= f  &&\textrm{ in } \Oe(t) ,
\\\label{eq:StrongStokesEvolutionaryDomain2}
\div(v_\e)&= 0 && \textrm{ in } \Oe(t) ,
\\\label{eq:StrongStokesEvolutionaryDomain3}
v_\e&= v_{\Ge} && \textrm{ in } \Ge(t),
\\\label{eq:StrongStokesEvolutionaryDomain4}
p_\e n - \e^2 \nu 2e(v_\e) n  &= p_b n&& \textrm{ in } \partial\Oe(t)  \setminus \Ge(t) 
\end{align}
for a force term $f$ on a time-dependent spatial domain $\Oe(t)$ with $t \in S$ for the time interval $S$. Thereby, $e(v_\e) \coloneqq (\nabla v_\e + \nabla v_\e^\top)/2 $ denotes the symmetric gradient of $v_\e$, which we use noting that in the standard derivation of the Stokes equation from the momentum balance equation, the continuity equation and the axioms of Newtonian fluids originally imply a symmetric stress tensor. Indeed, the incompressibility condition allows to replace $2e(v_\e)$ by $\nabla v_\e$ in the strong formulation. However, for the weak formulation, this substitution requires Dirichlet boundary values on the whole boundary, which is not the case in our model.
For the boundary condition, we distinguish between the interface of pore and solid space, which is denoted by $\Ge(t)$, and the remaining boundary.
At $\Ge(t)$, we use a Dirichlet boundary condition for the fluid velocity with boundary values $v_{\Ge}$.
This (inhomogeneous) Dirichlet boundary condition  \eqref{eq:StrongStokesEvolutionaryDomain4} is motivated by the no-slip boundary condition and allows a fluid velocity equal to the velocity of the boundary's deformation, which can be modelled by $v_\Ge$.
By using a normal stress boundary condition with outer unit normal vector $n$ and a normal boundary stress $p_b$ at $\partial \Oe(t) \setminus \Ge(t)$, we enable fluid in- and outflow at the boundary of the porous medium.
Thus, even if the total pore volume is changing, there does not arise an incompatibility with the fluid incompressibility so that we can consider this case as well.
In \cite{Mir16}, the homogenisation of Stokes flow with such a normal stress boundary condition at the outer boundary is considered for the case of a rigid domain and a homogeneous Dirichlet boundary condition at the pore interface.
	
We prove that the extension of $v_\e(t)$ by $0$ to $\Omega$ and the extension of the pressure $p_\e(t)$ by a cell-wise mean value converges weakly in $L^2(\Omega)$ for a.e.~$t \in S$ to the unique solution $(v(t),p(t))$ of the following Darcy law:
\begin{align}\label{eq:DarcysLawWithVolume1}
v\tx &= \frac{1}{\nu}K\tx(f\tx- \nabla p\tx) &&\textrm{ in } S \times \Omega,
\\\label{eq:DarcysLawWithVolume2}
\div(v\tx) &= -\int\limits_{\Gamma_x(t)} v_\Gamma\txy \cdot n dy \ \  ( = -\frac{d}{dt}\Theta\tx)  &&\textrm{ in } S \times \Omega,
\\
\label{eq:DarcysLawWithVolume3}
p &= p_b &&\textrm{ in } S \times \partial \Omega,
\end{align}
where $v_\Gamma$ is the two-scale limit of $v_\Ge$ and $\Gamma_x(t)$ is the interface of the pore and solid part in the reference cell at the macroscopic position $x$ at time $t$. If $v_\Ge$ is the velocity of the boundary deformation the right-hand side of \eqref{eq:DarcysLawWithVolume2} can be simplified to $-\frac{d}{dt}\Theta\tx$, where $\Theta$ is the porosity of the medium.
 
Compared to the Darcy law \eqref{eq:DarcysLaw}, the permeability tensor $K$ depends now on time and space taking into account the shape of a pore $\Ypx(t)$ at the point $x\in \Omega$ at the time $t\in S$. Moreover, the microscopic incompressibility condition becomes a macroscopic compressibility condition \eqref{eq:DarcysLawWithVolume2}. Combined with \eqref{eq:DarcysLawWithVolume1}, this gives an additional source or sink term for the pressure $p$. In the case that $v_\Ge$ is the velocity of the boundary deformation, this term captures the suction and suppression effects arising from the change of porosity.
	
For the homogenisation of \eqref{eq:StrongStokesEvolutionaryDomain1}--\eqref{eq:StrongStokesEvolutionaryDomain4}, we use the \textit{two-scale transformation method}. Thus, we transform the problem to a substitute problem onto a periodic reference domain $\Oe$ where we pass to the limit $\e$ to $0$ using two-scale convergence. Then, we transform the resulting limit problem back (cf.~\eqref{diagram}). This method was proposed for the homogenisation of a diffusion problem in \cite{Pet07}. Later it was applied in several  works -- in the same sense that the homogenisation of the substitute problem is proven -- (\cite{Pet07a},\cite{PB09}, \cite{EM17}, \cite{GNP21}). In \cite{Wie21} a rigorous two-scale convergence concept for this transformation method was developed and the method itself proven, i.e.~that \eqref{diagram} commutes. Thus, the homogenisation result for the periodic substitute problem can be transferred rigorously to a homogenisation result of the actual problem. 
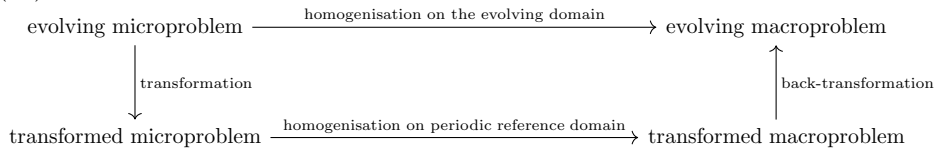
\begin{equation}\label{diagram}
\hspace{-0.1cm}
\begin{tikzcd}[scale cd=0.82, row sep=1cm, column sep = 4.9cm]
\textrm{evolving microproblem} \arrow[r, "\textrm{homogenisation on the evolving domain}"] \arrow[d, "\textrm{transformation}" ] & \textrm{evolving macroproblem} 
\arrow[d, leftarrow, "\textrm{back-transformation}"] \\
\textrm{transformed microproblem}
\arrow[r, "\textrm{homogenisation on periodic reference domain}"]& \textrm{transformed macroproblem}
\end{tikzcd}
\end{equation}

After the transformation onto the periodic reference domain, we derive uniform a priori estimates for the velocity field $v_\e$ and the pressure field $p_\e$. 
However, the transformation of the equation induces coefficients in the symmetric gradient. Therefore, we derive a uniform Korn-type inequality for the two-scale transformation method, which allows to estimate the transformed symmetric gradient from below.

Since the extension of the inhomogeneous Dirichlet boundary condition inside the domain is not necessarily divergence-free, we cannot estimate the velocity directly without estimating pressure as is done in the existing works on the homogenisation of Stokes flow.
Instead, we develop a family of $\e$-scaled operators $\div_\e^{-1}$, which are right inverse to the corresponding divergences, using the restriction operator which was introduced in \cite{Tar79} and developed further in \cite{All89}. Using these, we can deduce an $\e$-independent estimate on the velocity and the pressure during the existence proof without having an a priori estimate on the velocity at hand. 
For the case of a rectangular macroscopic domain with periodic Dirichlet boundary conditions, such a family of operators was constructed by V.~V.~Zhikov by a different idea (cf.~Ref.~\cite{Zhi94}). 

Having obtained these uniform a priori estimates, we can pass to the limit $\e \to 0$ in the reference configuration. There, we prove the strong convergence of the extension of the pressure. Then, we use two-scale compactness results in order to derive a microscopic incompressibility condition and a macroscopic compressibility condition. In the last step of the passage to the limit $\e \to 0$, we homogenise \eqref{eq:StrongStokesEvolutionaryDomain1} for divergence-free functions and reconstruct a microscopic pressure. The intermediate result is a two-pressure Stokes system in the cylindrical two-scale domain.

Afterwards, we transform this two-pressure Stokes system back into the reference configuration. Since the system contains not only microscopic derivatives but also macroscopic derivatives, it does not yield a transformation-independent result directly. The same problem occurs also by the formal back-transformation after the homogenisation of diffusion and elasticity equations in a periodic reference domain (cf.~\cite{Pet07}, \cite{EM17}). Using the results of \cite{Wie21}, the back-transformation can be done rigorously and a transformation-independent two-scale limit problem can be derived. We develop the ideas of \cite{Wie21} further and show how the back-transformed two-pressure Stokes system can be written independent of the transformation. In the last step, we separate the microscopic from the macroscopic scale and derive thereby the \textit{Darcy's law for evolving microstructure}.
This Darcy's law is different from the standard Darcy's law by its time- and space-dependent permeability tensor, which corresponds to the time- and space-dependent microscopic porosity. Moreover, it contains a new source term for the pressure equation. This term captures the suction and compression effects arising from the change of the porosity.

The paper is organized as follows: in Section \ref{section:Transformation}, we transform the Stokes problems onto the periodic domain $\Oe$ by the two-scale transformation method. In Section \ref{section:AprioriEstimates}, we derive a uniform Korn inequality for the two-scale transformation method and a family of $\e$-scaled operators $\div_\e^{-1}$, which are right inverse to the corresponding divergences. Using these results, we give uniform a priori estimates, which allow to pass to the homogenisation limit.
In Section \ref{section:Homogenisation}, we pass to the limit $\e \to 0$ in the reference configuration and derive a two-pressure Stokes equation.
In Section \ref{section:BackTrafo}, we transform this two-scale limit problem back to the actual domain and derive \eqref{eq:DarcysLawWithVolume1}--\eqref{eq:DarcysLawWithVolume3}, which we call Darcy's law for evolving microstructure.

\begin{nota}
In the following, let $C \in \R$ be a generic constant independent of $t \in S$ and $\e >0$. Moreover, let $1 \leq p_s \leq \infty $ and $C(t)$ be a generic time-dependent function independent of $\e$, where $C \in L^{p_s}(S)$

For an open set $U \subset \R^N$, we write $(f,g)_U$ for the scalar product of $f,g\in L^2(U)$ and define $\norm{f}{U} \coloneqq (f,f)_U^{\frac{1}{2}}$.
If $G \subset \partial U$ is Lipschitz regular, we define $H^1_G(U)\coloneqq \{v \in H^1(U) \mid v|_G = 0 \}$. Moreover, we define $H^1_{\Gamma\#}(\Yp) \coloneqq \{v \in H^1_\Gamma(\Yp) \mid  v \textrm{ is } Y\textrm{-periodic} \}$ for an open Lipschitz set $\Yp \subset (0,1)^N$ and $\Gamma = \partial \Yp \cap \partial (Y \setminus \Yp)$.
In the following, we identify $H^1_{\Gamma\#}(\Yp)$ with $\{v \in H^1_\#(Y) \mid v|_{Y \setminus \Yp} = 0 \}$ if the function has to be defined on the whole of $Y$.
\end{nota}

\section{The Stokes problem on the microscopic scale}\label{section:Transformation}
Let $S = (0,T)$ be the finite time interval. Let $\Omega \subset \R^N$ be a bounded and connected domain which can be represented as a finite union of axis-parallel cuboids with corner coordinates in $\Q^N$, representing the macroscopic domain of the porous medium. Thus, there exists a sequence $(\e_n)_{n \in \N}$ such that $\Omega = \operatorname{int}\Big(\bigcup_{k \in I_{\e_n} } k + \e_n \overline{Y}\Big)$ for every $n\in \N$, where $I_\e \coloneqq\{k \in \e\Z^N \mid \operatorname{int}(k + \e Y) \subset \Omega \}$ for $\e >0$ and  $Y \coloneqq (0,1)^{N\times N}$ is the microscopic reference cell. We consider in the following such a fixed sequence $(\e_n)_{n \in \N}$ with $0 <\e_n <1$ for every $n \in \N$ and write shortly $\e$.
Let $\Yp \subset Y$ be open and $Y^\mathrm{s} \coloneqq \operatorname{int}(Y \setminus \Yp)$ such that
\begin{enumerate}
\item $\Yp$ and $Y^\mathrm{s}$ have positive measure,
\item $\Yp$ is a connected set with Lipschitz boundary,
\item $\partial \Yp \cap \{x_i = 0\} +e_i = \partial \Yp \cap \{x_i = 1\}$ for every $i \in\{1, \dots, N \}$, 
\item $Y^p_\# \coloneqq \operatorname{int}\Big(\bigcup\limits_{k \in \Z^N} k + \overline{\Yp}\Big)$ is connected and has a $C^1$-boundary.
\end{enumerate}
These sets denote the pore part and the solid part of the reference cell. We denote the interface of them by $\Gamma \coloneqq \partial \Yp \cap \partial Y^\mathrm{s}$.

We define $\Oe$, which represents the pore part in the reference configuration, by
$\Oe  \coloneqq \operatorname{int}\Big(\bigcup_{k \in I_\e}  k + \e \overline{\Yp}  \Big)$,
the corresponding solid part by 
$\Oes  \coloneqq \operatorname{int}\Big(\Omega \setminus \Oe \Big)$ and the interface of these by $\Ge = \partial \Oe \cap \partial \Oes$.
Note that $\Oe$ is connected and $\Ge$ as well as the remaining part of the boundary $\partial \Oe \setminus \Ge$ are Lipschitz regular by their construction.

We assume that, for a.e.~$t\in S$, the evolving domain $\Oe(t)$ and the evolving surface $\Ge$ can be described by locally periodic transformations $\psi_\e \in L^\infty(S;C^1(\overline{\Omega})^N)$. That means $\Oe(t) = \psi_\e(t,\Oe)$, $\Oes(t) = \psi_\e(t,\Oes)$, $\Ge(t) = \psi_\e(t,\Ge)$ for a.e.~$t\in S$.

We denote the Jacobians of $\psi_\e$ with respect to $x$ by $\Psi_\e\tx = D \psi_\e\tx$ and $J_\e\tx \coloneqq \det(\Psi_\e\tx)$.
\newpage
\begin{ass}\label{ass:psi}We assume that 
\begin{enumerate}
\item $\psi_\e(t,\cdot_x)$ is a $C^2$-diffeomorphism from $\overline{\Omega}$ onto $\overline{\Omega}$ for a.e.~$t\in S$ with inverse $\psi_\e^{-1}(t,\cdot_x)$ for $\psi_\e, \psi_\e^{-1} \in L^\infty(S;C^2(\overline{\Omega}))$ 
\item there exists $c_J >0$ such that $J_\e(t) \geq c_J$ for a.e.~$t\in S$,
\item there exists $C>0$ such that $\e^{i-1}\norm{ \check{\psi}_\e}{L^\infty(S;C^i(\overline{\Omega}))}\leq C$ for $i \in \{0,1,2\}$, where $\check{\psi}_\e(x) \coloneqq \psi_\e(x) -x$ are the corresponding displacement mappings,
	
\item there exists $\psi_0 \in L^\infty(S \times \Omega;C^2(\overline{Y}))^N$ such that 
\begin{enumerate}
\item
$\psi_0(t,x,\cdot_y) : \overline{Y} \rightarrow  \overline{Y}$ are $C^2$-diffeomorphisms for a.e.~$\tx\in S\times \Omega$ with inverses $\psi_0^{-1}(t,x,\cdot_y)$ for $\psi_0, \psi_0^{-1} \in L^\infty(S\times \Omega;C^2(\overline{Y}))$,
\item the corresponding displacement mapping $\check{\psi}_\e(t,x,y) \coloneqq \psi_\e(t,x,y) -y$, can be extended $Y$-periodically such that $\check{\psi} \in L^\infty(\Omega;C^2_\#(\overline{Y}))$,
\item $\e^{|\alpha| -1} D_{x_\alpha}\check{\psi}_\e(t)$ two-scale converges strongly with respect to every $L^p$-norm for $p\in (1,\infty)$ to $D_{y_\alpha}\check{\psi}_0(t)$ for every multiindex $\alpha \in \{0,1,2\}^{N}$ with $|\alpha| \leq 2$.
\end{enumerate}
\end{enumerate}
\end{ass}

An evolution of the domains which satisfy Assumption \ref{ass:psi} can be obtained for example from the following model:
\begin{exam}
Let $\Theta : [0,T] \times \overline{\Omega} \rightarrow (0,1)$ be a smooth function with $D \Theta$ small enough (which describes for example the local porosity) and let $\psi_0 : (0,1) \times \Yp \rightarrow \overline{Y}$	be a smooth mapping such that $\psi_0(\Theta, \Yp)$ gives a cell with porosity $\Theta$ and the corresponding displacement mapping $\tilde{\psi}_0(\Theta,y) = \psi_0(\Theta, y)-y$ can be extended to a $Y$-periodic function.
Then,
\begin{align*}
\psi_\e \tx \coloneqq x + \e \tilde{\psi}_0\left(\Theta\tx, \frac{x}{\e}\right), \hspace{1cm} \psi_0 \txy \coloneqq \tilde{\psi}_0(\Theta\tx, y)
\end{align*}
fulfil Assumptions \ref{ass:psi}. 
\end{exam}

\begin{ass}\label{ass:Data}
Let $p_s \in [1,\infty]$ be fixed. We assume on the data that:
\begin{enumerate}
\item $f_\e$ is a sequence in $L^{p_s}(S;L^2(\Omega))$ such that $\norm{f_\e(t)}{L^2(\Omega)} \leq C(t)$ for a.e.~$t\in S$ for $C\in L^{p_s}(S)$
\item there exists $f\in L^{p_s}(S;L^2(\Omega))$ such that 
$f_\e(t)$ two-scale converges weakly with respect to the $L^2$-norm to $f(t)$ for a.e.~$t\in S$,
\item $p_{b,\e}$ is a sequence in $L^{p_s}(S;H^1(\Omega))$ such that $\norm{p_{b,\e}(t)}{L^2(\Omega)} \leq C(t)$ for a.e.~$t\in S$ for $C\in L^{p_s}(S)$,
\item there exists $(p_b,p_{b,1})\in L^{p_s}(S;H^1(\Omega)) \times L^{p_s}(S;L^2(\Omega;H^1_\#(Y)/\R))$ such that 
$\nabla p_{b,\e}(t)$ two-scale converges weakly with respect to the $L^2$-norm to $\nabla_x p_b(t) + \nabla_y p_{b,1}(t)$ for a.e.~$t\in S$,
\item $v_\Ge$ is a sequence in $L^{p_s}(S;H^1(\Omega))$ such that $\frac{1}{\e}\norm{v_\Ge(t)}{L^2(\Omega)} \leq C(t)$ for a.e.~$t\in S$ and $\norm{\nabla v_\Ge(t)}{L^2(\Omega)} \leq C(t)$ for a.e.~$t\in S$ for $C\in L^{p_s}(S)$, \color{black}
\item there exists $v_\Gamma \in L^{p_s}(S;L^2(\Omega;H^1_\#(Y)))$ such that $\frac{1}{\e} v_\Ge(t)$ two-scale converges weakly with respect to the $L^2$-norm to $v_\Gamma(t)$ for a.e.~$t\in S$ and $\nabla v_\Ge(t)$ two-scale converges weakly with respect to the $L^2$-norm to $\nabla_y v_\Gamma(t)$  for a.e.~$t\in S$, 

\item if $v_\Ge$ should be the velocity of the boundary deformation, i.e.~$v_\Ge\tx = \partial_t \psi_\e(t, \psi_\e^{-1}\tx )$, we assume that $\psi_\e $ is a sequence in $W^{1,p_s}(S;H^1(\Omega))$ such that $\frac{1}{\e}\norm{\partial_t \psi_\e(t)}{L^2(\Omega)} \leq C(t)$ for a.e.~$t \in S$ for $C \in L^{p_s}(S)$. Moreover, we assume that $\psi_0 \in W^{1,p_s}(S;L^2(\Omega;H^1_\#(Y)))$ such that $\e^{|\alpha|-1}D_{x_\alpha} \partial_t \psi_\e(t)$ two-scale converges weakly with respect to the $L^2$-norm to $D_{y_\alpha} \partial_t \psi_0(t)$ for a.e.~$t\in S$ and every multiindex $\alpha \in \{0,1\}^N$ with $|\alpha|\leq 1$.
\end{enumerate}
\end{ass}

Since we consider the stationary Stokes equations, time becomes only a parameter. Therefore, we have formulated the previous assumptions in a way that allows us to consider the equation and the homogenisation process pointwise in time for a.e.~$t \in S = (0,T)$. However, we have assumed for all used quantities the measurability with respect to time. Thus, we can show that the solutions of the $\e$-scaled problem can be uniformly bounded for a.e.~$t \in S$ by a $L^{p_s}(S)$ bound for a fixed $p_s \in [1,\infty]$, which allows to translate the two-scale convergence into the time-dependent two-scale convergence, which is used in parabolic problems. This allows a coupling of the Stokes problem with such process in future works.

In order to derive a weak form for \eqref{eq:StrongStokesEvolutionaryDomain1}--\eqref{eq:StrongStokesEvolutionaryDomain2}, we substitute the boundary values and define
$w_\e(t) \coloneqq v_\e(t) - v_\Ge(t)$ and $q_\e(t) \coloneqq p_\e(t) - p_{b,\e}(t)$.
Then, we multiply \eqref{eq:StrongStokesEvolutionaryDomain1} by test functions $\varphi$ which are $0$ on $\Ge(t)$ and integrate over $\Oe(t)$. After integration by parts and using the boundary conditions \eqref{eq:StrongStokesEvolutionaryDomain3}--\eqref{eq:StrongStokesEvolutionaryDomain4}, we get \eqref{eq:WeakFormUntransformed1}.
In addition, we multiply \eqref{eq:StrongStokesEvolutionaryDomain2} with a test function $\phi \in L^2(\Oe(t))$ and integrate over $\Oe(t)$. Then, we obtain the following weak form:

Find $(w_\e,q_\e) \in L^{p_s}(S; H^1_{\Ge(t)}(\Oe(t)))  \times  L^{p_s}(S; L^2(\Oe(t)))$ such that, for a.e.~$t\in S$,
\begin{align}\notag
\int\limits_{\Oe(t)} \nu\e^2 2e(w_\e\tx) : \nabla \varphi(x) - q_\e\tx \div(\varphi(x)) = \int\limits_{\Oe(t)} f_\e \tx \cdot \varphi(x) dx
\\\label{eq:WeakFormUntransformed1}
- 
\int\limits_{\Oe(t)} \nu\e^2 2 e(v_\Ge\tx) : \nabla \varphi(x) +\nabla p_{b,\e}\tx \cdot \varphi(x) dx
\\\label{eq:WeakFormUntransformed2}
\int\limits_{\Oe(t)}\div(w_\e\tx) \phi(x) dx
=
-\int\limits_{\Oe(t)}\div(v_\Ge\tx) \phi(x) dx
\end{align}
for every $(\varphi, \phi) \in H^1_{\Ge(t)}(\Oe(t))^N \times L^2(\Oe(t))$.

\subsection{Transformation on the periodic reference domain}
We transform the given data on the reference configuration by
\begin{align}\label{eq:TrafoData}
\hat{f}_\e\tx \coloneqq f(t,\psi_\e\tx), \ \hat{p}_{b,\e}(x,t) \coloneqq p_{b,\e}(t,\psi_\e\tx), \ \hat{v}_{\Ge}(x,t) \coloneqq v_{\Ge}(t,\psi_\e\tx)
\end{align}
and define $A_\e = J_\e \Psi_\e^{-1}$ as well as the transformed symmetric gradient $\hat{e}_{\e,t}(v)  \coloneqq (\Psi_\e^{-\top}(t) \nabla v + (\Psi_\e^{-\top}(t) \nabla v)^\top )/2$.
Then, we transform  \eqref{eq:WeakFormTransformed1}--\eqref{eq:WeakFormTransformed2} onto the periodic reference domain and obtain the following problem:

Find $\hat{w}_\e,\hat{q}_\e \in L^{p_s}(S; H^1_\Ge(\Oe))  \times  L^{p_s}(S; L^2(\Oe))$ such that for a.e.~$t\in S$
\begin{align}\notag
\intOe \nu\e^2 A_\e\tx2\hat{e}_{\e,t}(\hat{w}_\e\tx) : \nabla \varphi(x) - \hat{q}_\e\tx \div(A_\e(t,x)\varphi(x)) 
\\\notag
= \intOe J_\e\tx\hat{f}_\e\tx \cdot \varphi(x) dx- 
\intOe \nu\e^2 A_\e\tx 2 \hat{e}_{\e,t}(\hat{v}_\Ge\tx) : \nabla \varphi(x) dx
\\\label{eq:WeakFormTransformed1}
- 
\intOe A_\e^\top\tx\nabla \hat{p}_{b,\e}\tx \cdot \varphi(x) dx
\\\label{eq:WeakFormTransformed2}
\intOe\div(A_\e\tx\hat{w}_\e\tx) \phi(x) dx
=
-\intOe\div(A_\e\tx\hat{v}_\Ge\tx) \phi(x) dx
\end{align}
for every $(\varphi, \phi) \in H^1_{\Ge}(\Oe)^N \times L^2(\Oe)$.

\begin{lemma}\label{lemma:EquivalenceTransformation}
Problem \eqref{eq:WeakFormUntransformed1}--\eqref{eq:WeakFormUntransformed2} is equivalent to  \eqref{eq:WeakFormTransformed1}--\eqref{eq:WeakFormTransformed2}, in the sense that $(w_\e,q_\e)$ solves \eqref{eq:WeakFormUntransformed1}--\eqref{eq:WeakFormUntransformed2} if and only if $(\hat{w}_\e,\hat{q}_\e)$ solves \eqref{eq:WeakFormTransformed1}--\eqref{eq:WeakFormTransformed2}, where the solutions can be transformed by $\hat{w}_\e\tx = w_\e(t,\psi_\e\tx)$ and  $\hat{q}_\e\tx = q_\e(t,\psi_\e\tx)$.
\end{lemma}
\begin{proof}
Using the product rule, we can transform between \eqref{eq:WeakFormUntransformed1}--\eqref{eq:WeakFormUntransformed2} and  \eqref{eq:WeakFormTransformed1}--\eqref{eq:WeakFormTransformed2}.
\end{proof}

\section{Existence and uniform a priori estimates}\label{section:AprioriEstimates}
In this section, we show the following existence and uniqueness result for \eqref{eq:WeakFormTransformed1}--\eqref{eq:WeakFormTransformed2} and derive an $\e$-independent bound for the solution.
\begin{theorem}\label{thm:ExistenceEps}
For given $\e>0$, Problem \eqref{eq:WeakFormTransformed1}--\eqref{eq:WeakFormTransformed2} has a unique solution  $(\hat{w}_\e, \hat{q}_\e) \in  L^p(S;\HepsN)$ $\times L^p(S;L^2(\Omega_\e))$
such that 
\begin{align}\label{eq:EpsIndependentEstimateVP}
\norm{\hat{w}_\e(t)}{\Omega_\e}
+ \e \norm{\nabla \hat{w}_\e(t)}{\Omega_\e}
+
\norm{\hat{q}_\e(t)}{\Oe}
\leq C(t)
\end{align}
for a.e.~$t \in S$
and $C \in L^{p_s}(S)$.
\end{theorem}
For the proof of Theorem \ref{thm:ExistenceEps}, we use the following generic saddle-point formulation, where $V$ and $Q$ are Hilbert spaces and $a: V \times V \rightarrow \R$ and $b: V \times Q \rightarrow \R$ continuous bilinear forms:

Given $f \in V'$ and $g \in Q'$, find  a solution 
$(v,p) \in V \times Q$ such that:
\begin{align}\label{eq:InfSupV}
a(v,\varphi) + b(\varphi,p) = \langle f,\varphi \rangle_{V' \times V} \qquad \ \textrm{ for all } \varphi \in V,
\\\label{eq:InfSupP}
b(v,\phi) = \langle g,\phi \rangle_{Q' \times Q} \qquad  \ \textrm{ for all } \phi \in Q.
\end{align}

The existence and uniqueness of a solution and a corresponding estimate for such saddle-point problems are given by the following well-known lemma. A proof is given in \cite[Theorem 4.2.3]{BF10}, for example.
\begin{lemma}\label{lemma:SaddlePoint}
If there exist constants $\alpha,\beta > 0$ such that:
\begin{align}
&a(w,w) \geq \alpha \norm{w}{V}^2 \ \ \textrm{ for all } w \in V,
\\
&\inf\limits_{u \in Q} \sup\limits_{w \in V} \frac{|b(w,u)|} {\norm{w}{V} \norm{u}{Q}} \geq \beta,
\end{align}
then the saddle-point problem \eqref{eq:InfSupV}--\eqref{eq:InfSupP} has a unique solution $(v,p)\in V \times Q$.
Furthermore, the following estimates hold for the solution:
\begin{align}\label{eq:estimateV}
\norm{v}{V} &\leq \frac{1}{\alpha} \norm{f}{V'} + \frac{2\norm{a}{}}{\alpha\beta }\norm{g}{Q'}, 
\\\label{eq:estimateP}
\norm{p}{Q} &\leq\frac{2\norm{a}{}}{\alpha\beta } \norm{f}{V'} +
\frac{2\norm{a}{}^2}{\alpha \beta^2 }\norm{g}{Q'},
\end{align}
where $\norm{a}{} \coloneqq \sup\limits_{v,w \in V} \frac{|a(v,w)|}{\norm{v}{V} \norm{w}{V}}$ and $\norm{b}{} \coloneqq \sup\limits_{v\in V, p \in Q} \frac{|b(v,p)|}{\norm{v}{V} \norm{p}{Q}}$ are the norms on the space of continuous bilinear forms which are given by the continuity constants.	
\end{lemma}
The following Lemma enables us to add time as a parameter in Lemma \ref{lemma:SaddlePoint}. More precisely, we use it later to show that $(w_\e,q_\e)$ is measurable with respect to time.
\begin{lemma}\label{lemma:ContinuityWithRespectToData}
For the spaces
\begin{align*}
&A \coloneqq \{a \in \mathrm{Bil}(V,V) \mid a(v,v) \geq \alpha \norm{v}{V}^2 \   \textrm{ for all } v \in V \textrm{ and }  \alpha >0  \},
\\
&B \coloneqq \left\{b \in \mathrm{Bil}(V,Q) \mid  \inf\limits_{p \in Q} \sup\limits_{v \in V} \frac{|b(v,p)|} {\norm{v}{V} \norm{p}{Q}} \geq \beta \textrm{ for } \beta >0 \right\} 
\end{align*}
of bilinear forms with their induced norms, which are defined in Lemma \ref{lemma:SaddlePoint},
the unique solution of the corresponding saddle-point problem \eqref{eq:InfSupV}--\eqref{eq:InfSupP} given by Lemma \ref{lemma:SaddlePoint} depends continuously on the data $(a,b,f,g) \in A \times B \times V' \times Q'$.
\end{lemma}
\begin{proof}
Lemma \ref{lemma:ContinuityWithRespectToData} can be proven by computations which are similar to those in standard proofs of the estimates  of Lemma \ref{lemma:SaddlePoint}.
\end{proof}
In order to derive the uniform estimate \eqref{eq:EpsIndependentEstimateVP}, we employ \eqref{eq:estimateV} and \eqref{eq:estimateP}. Hence, we equip  $\HepsN$ with a proper norm and  derive a uniform coercivity and a uniform inf--sup estimate for the bilinear forms.

First, we show some uniform estimates for the coefficients (cf.~Lemma \ref{lemma:EstimatesPsi}).
Then, we derive a family of $\e$-scaled Korn-type inequalities for the two-scale transformation method (cf.~Lemma \ref{lemma:KornPsiEps}). These Korn-type inequalities allow to estimate the transformed symmetric gradients $\hat{e}_{\e,t}(\hat{w_\e})$ uniformly from below, which implies the uniform coercivity for the first bilinear form.
In order to show the uniform inf--sup estimate for the other bilinear form, we construct a family of $\e$-scaled operators $\div^{-1}_\e$, which are right inverses to the divergence operator (cf.~Lemma \ref{lemma:dive-1}).

\begin{lemma}\label{lemma:EstimatesPsi}
There exists a constant $C>0$ such that
\begin{align*}
&\norm{J_\e}{L^\infty(S;C(\overline{\Oe}))} + \norm{\Psi_\e}{L^\infty(S;C(\overline{\Oe}))} + \norm{\Psi_\e^{-1}}{L^\infty(S;C(\overline{\Oe}))} \leq C,
\\
&\e\norm{\partial_{x_i}J_\e}{L^\infty(S;C(\overline{\Oe}))}
+
\e\norm{\partial_{x_i}J_\e^{-1}}{L^\infty(S;C(\overline{\Oe}))}
\leq C,
\\
&\e\norm{\partial_{x_i}\Psi_\e}{L^\infty(S;C(\overline{\Oe}))}
+\e\norm{\partial_{x_i}A_\e}{L^\infty(S;C(\overline{\Oe}))} + \e\norm{\partial_{x_i}\Psi_\e^{-1}}{L^\infty(S;C(\overline{\Oe}))} \leq C
\end{align*}
for every $i\in \{0, \dots,N\}$.
\end{lemma}
\begin{proof}
We note that $\Psi_\e = \1 +D \check{\psi}_\e$. Then the uniform estimate of $D\check{\psi}_\e$ given by Assumption \ref{ass:psi} shows that $\norm{\Psi_\e}{L^\infty(S;C(\overline{\Oe}))}\leq C$.

Since $J_\e$ and the entries of $\Psi_\e$ are polynomials with respect to the entries of $\Psi_\e$ and $J_\e^{-1}$, the uniform bound of $J_\e\geq c_J>0$ from below (cf.~Assumption \ref{ass:psi}) and $\norm{\Psi_\e}{L^\infty(S;C(\overline{\Oe}))}\leq C$ implies $\norm{J_\e}{L^\infty(S;C(\overline{\Oe}))} \leq C$ and $\norm{\Psi_\e^{-1}}{L^\infty(S;C(\overline{\Oe}))}\leq C$.

After rewriting $\partial_{x_i}\Psi_\e = \partial_{x_i}D \psi_\e = \partial_{x_i}(\1 + D \check{\psi}_\e)= \partial_{x_i}D\check{\psi}_\e$, Assumption~\ref{ass:psi} shows  $\e\norm{\partial_{x_i}\Psi_\e}{L^\infty(S;C(\overline{\Oe}))} \leq C$ for every $i\in \{0, \dots,N\}$.

We note that $A_\e$ is the adjugate matrix of $\Psi_\e = \1 + D \psi_\e$. Thus, all of its entries are minors of $\Psi_\e$. We rewrite the $x_i$-derivative of these minors with the product rule into the sum of products, where each product has $(n-2)$-factors which are entries of $\Psi_\e(t)$ and one factor which is a entry of $\partial_{x_i}\Psi_\e$.
Then, the estimates $\norm{\Psi_\e}{L^\infty(S;C(\overline{\Oe}))}\leq C$ and $\e\norm{\partial_{x_i}\Psi_\e}{L^\infty(S;C(\overline{\Oe}))} \leq C$ give $\e\norm{\partial_{x_i}A_\e}{L^\infty(S;C(\overline{\Oe}))} \leq C$.

We obtain $\e\norm{\partial_{x_i}J_\e}{L^\infty(S;C(\overline{\Oe}))} \leq C$ by the same argumentation as for the estimate of $\e\norm{\partial_{x_i}A_\e}{L^\infty(S;C(\overline{\Oe}))}$. 

Using the chain rule, we rewrite $\partial_{x_i}J_\e^{-1} = -J_\e^{-2}\partial_{x_i}J_\e$. Then, the uniform bound $J_\e \geq c_J>0$ from below and the estimate $\e\norm{\partial_{x_i}J_\e}{L^\infty(S;C(\overline{\Oe}))} \leq C$ imply the estimate $\e\norm{\partial_{x_i}J_\e^{-1}}{L^\infty(S;C(\overline{\Oe}))} \leq C$

We rewrite $\Psi_\e^{-1} = J_\e^{-1} A_\e$. Then,  we obtain  $\e\norm{\partial_{x_i}\Psi_\e^{-1}}{L^\infty(S;C(\overline{\Oe}))} \leq C$ with the product rule and the previous estimates.
\end{proof}
\subsection{Korn-type inequality for the two-scale transformation method}
In order to derive the Korn-type inequalities for the two-scale transformation method, we need the following $\e$-scaled Poincar\'e inequality for periodic domains.
\begin{lemma}\label{lemma:PoincareEps}
There exists a constant $C_P \in \R$  such that
\begin{align*}
\norm{v}{\Oe} \leq  \e C_P\norm{\nabla v}{\Oe}
\end{align*}
for every $v \in \HepsN$.
\end{lemma}
\begin{proof}
Lemma \ref{lemma:PoincareEps} is a standard result and can be proven by covering $\Oe$ with $\e$-scaled copies of $\Yp$ and  scaling them on $\Yp$. Then, applying the Poincar\'e inequality for piecewise zero boundary values there and scaling back yields the estimate.
\end{proof}

\begin{lemma}\label{lemma:KornPsiEps}
There exists a constant $\alpha \in \R$ independent of $\e$, such that 
\begin{align}
\alpha \norm{\nabla v}{\Oe}^2 \leq  \norm{\hat{e}_{\e,t}(v)}{\Oe}^2
\end{align}
for $\e >0 $ and all $v \in \HepsN$.
\end{lemma}
\begin{proof}
Lemma \ref{lemma:KornPsiEps} follows directly from Lemma \ref{lemma:KornEpsCoeff}, since the estimates of Lemma \ref{lemma:EstimatesPsi} shows that the prerequisites of Lemma \ref{lemma:KornEpsCoeff} are fulfilled.
\end{proof}

\begin{lemma}\label{lemma:KornEpsCoeff}
Let $c,C >0$. Then, there exists an $\e$-independent constant $\alpha>0$ such that 
\begin{align*}
\alpha \norm{\nabla v}{\Oe}^2 \leq \norm{M_\e \nabla v + \left(M_\e \nabla v\right)^\top }{\Oe}^2
\end{align*}
for any $v \in H^1_{\Gamma_\e}(\Oe)^N$
and
for every
$M_\e \in C^{0,1}(\overline{\Oe}) $ with
\begin{align*}
&\norm{M_\e}{C(\overline{\Oe})} + \e \norm{M_\e}{C^{0,1}(\overline{\Oe})}\leq C,
\\
&\det(M_\e(x)) \geq c >0 \textrm{ for every } x \in \overline{\Oe}.
\end{align*}
\end{lemma}

\begin{proof}
Since $\norm{M_\e \nabla v + (M_\e \nabla v)^\top}{\Oe}^2 = \sum\limits_{k \in I_\e} \norm{M_\e \nabla v + (M_\e \nabla v)^\top}{k + \e \Yp}^2$ and $\norm{\nabla v}{\Oe}^2 = \sum\limits_{k \in I_\e} \norm{\nabla v}{k + \e \Yp}^2$, we can reduce the problem on the reference cell.
After transforming $k + \e \Yp$ on $\Yp$ it is sufficient to show 
\begin{align}
\alpha \norm{v}{\Yp}^2 \leq \norm{ M \nabla v + \left(M \nabla v\right)^\top   }{\Yp}^2
\end{align}
for any $v \in H^1_{\Gamma}(\Yp)^N$, $\e >0$ and $k \in I_\e$ where $M(x) \coloneqq M_\e(k + \e x)$.
From the Lipschitz continuities of $M_\e$ and the transformation $x \mapsto k + \e x$, we can conclude that $\norm{M}{C^{0,1}(\overline{\Yp})} \leq C$.
The uniform bound of the determinant from below remains preserved under the transformation.
Hence $M \in \mathcal{M} \coloneqq \{M \in  C^{0,1}(\overline{\Yp})^{N \times N} \mid  \norm{M}{C^{0,1}(\overline{\Yp})} \leq C  \textrm{ and }  \det(M) \geq c \}$.
	
The uniform Lipschitz continuity of $\mathcal{M}$ implies the equicontinuity of $\mathcal{M}$ and since $\mathcal{M}$ is also pointwise bounded, we obtain by the theorem of Arzel\`a--Ascoli that $\mathcal{M}$ is relatively compact in $C(\overline{\Yp})^{N \times N}$. 
Then, we apply Lemma \ref{lemma:KornUniformCoeff} on the closure of $\mathcal{M}$ in $C(\overline{\Yp})^{N \times N}$ and we obtain Lemma \eqref{lemma:KornEpsCoeff}.	
\end{proof}

\begin{lemma}\label{lemma:KornUniformCoeff}
Let $U \subset \R^N$ be an open connected Lipschitz domain and $G \subset \partial U$ open with $|G| >0$.
Let $\mathcal{M}$ be a compact subset of $C(\overline{U})^{N \times N}$ and assume there exists $c>0$ such that $\det(M) \geq c$ for every $M \in \mathcal{M}$.
Then, there exists $\alpha >0$ such that 
\begin{align*}
\alpha \norm{\nabla v}{U} \leq \norm{(M \nabla v +  (M\nabla v)^\top}{U}
\end{align*}
for every $M \in \mathcal{M}$ and $v \in H^1_G(U)$.
\end{lemma}
\begin{proof}
Let $M \in \mathcal{M}$, then Lemma \ref{lemma:KornCoeff} gives a constant $\alpha_M$ such that
\begin{align}\label{eq:KornA}
\alpha_M\norm{\nabla v}{U} \leq  \norm{M \nabla v + (M\nabla v)^\top}{U}
\end{align}	for every $v \in H^1_G(U)$.
We obtain for $B \in \mathcal{M}$
\begin{align*}
\norm{M \nabla v + (M\nabla v )^\top  - (B \nabla v + (B \nabla v)^\top )}{U}
\leq 2 \norm{M-B}{C(\overline{U})} \norm{\nabla v}{U},
\end{align*}
which implies
\begin{align}\label{eq:EstAB}
\norm{M \nabla v + (M\nabla v)^\top}{U} \leq \norm{B \nabla v + (B\nabla v )^\top}{U} +  2\norm{M-B}{C(\overline{U})} \norm{\nabla v}{U}.
\end{align}
Combining \eqref{eq:KornA} and \eqref{eq:EstAB} gives for any $B \in B_{\alpha_M/4}(M)$
\begin{align}\label{eq:EstSymB}
\frac{1}{2}\alpha_M\norm{\nabla v}{U} \leq \norm{B \nabla v + (B\nabla v )^\top}{U}.
\end{align}
Then, we cover $\mathcal{M}$ by $\bigcup\limits_{M \in \mathcal{M}} B_{\alpha_M/4}(M)$ and since $\mathcal{M}$ is compact, there exists a finite set $\mathcal{I}$  such that for every $B \in \mathcal{M}$ there exists $M\in \mathcal{I}$ with $B \in B_{\alpha_M/4}(M)$.
We choose $\alpha = \min\limits_{M \in \mathcal{I}} \alpha_M/2$ and obtain from \eqref{eq:EstSymB}
\begin{align*}
\alpha \norm{\nabla v}{U} \leq \frac{\alpha_M}{2} \norm{\nabla v}{U} \leq \norm{(B \nabla v + (B\nabla v)^\top}{U}
\end{align*}
for every $B \in \mathcal{M}$ and $v \in H^1_G(U)$.
\end{proof}

\begin{lemma}\label{lemma:KornCoeff}
Let $U \subset \R^N$ be an open connected Lipschitz domain and $G \subset \partial U$ open with $|G| >0$.
Let	$A : \overline{U} \rightarrow \R^{n \times n}$ be a continuous mapping with $\det(A) \geq c >0$. Then, there is a constant $\alpha >0$ such that 
\begin{align*}
\alpha \norm{\nabla v}{U} \leq \norm{(A \nabla v + (A\nabla v)^\top}{U}
\end{align*}	
for all $v \in H^1_G(U)$.
\end{lemma}
\begin{proof}
Lemma \ref{lemma:KornCoeff} is proven in \cite{Pom03}.
\end{proof}

\subsection{Right-inverse divergence operator}
In order to construct explicitly the operators $\div_\e^{-1} : L^2(\Omega)\rightarrow H^1_\Ge(\Oe)$, we use the following right inverse divergence operator (see.~Lemma \ref{lemma:div-1}) and the restriction operator (see.~Lemma \ref{lemma:Re}), which was originally introduced in \cite{Tar79} and developed further in \cite{All89}.
\begin{lemma}\label{lemma:div-1}
Let $U$ be a bounded, connected Lipschitz domain. Then, there exists a bounded linear operator $\div^{-1} : L^2(U) \rightarrow H^1(U)$ such that $\div \circ \div^{-1} = id_{L^2(U)}$.
\end{lemma}
\begin{proof}
Lemma \ref{lemma:div-1} can be easily proven by means of the Bogovski\v{i} operator.
\end{proof}
\begin{lemma}\label{lemma:Re}
There exists a linear continuous operator $R_\e : H^1(\Omega)^N \rightarrow H^1_\Ge(\Oe)^N$ such that 
\begin{enumerate}
\item $u \in  H^1_\Ge(\Oe)$ implies $R_\e u = u$ in $\Oe$
\item $\div (R_\e u) = \div(u) + \frac{1}{|\Yp|}\sum\limits_{k\in I_\e} \chi_{k+ \e\Yp}\int\limits_{k + \e\Ys} \div(u)$,
\item $\div(u) =0$ implies $\div(R_\e u) = 0$,
\item there exists a constant $C$ such that
\begin{align*}
\norm{R_\e u}{\Oe} + \e\norm{\nabla (R_\e u)}{\Oe} \leq \norm{ u}{\Omega} +\e \norm{\nabla u}{\Omega}
\end{align*}
for every $u \in H^1(\Omega)^N$.
\end{enumerate}
\end{lemma}
\begin{proof}
In \cite{All89} this restriction operator is explicitly constructed from $H^1_0(\Omega)^N$ to $H^1_0(\Oe)^N$. Indeed, the construction is done locally so that the same construction yields an operator $R_\e : H^1(\Omega)^N \rightarrow H^1_\Ge(\Oe)^N$.
\end{proof}
\begin{lemma}\label{lemma:dive-1}
There exists a linear continuous operator $\div_\e^{-1} : L^2(\Oe) \rightarrow H^1_\Ge(\Oe)^N$, which is right inverse to the divergence, i.e.~$\div \circ \div_\e^{-1} = id_{L^2(\Oe)}$, such that
\begin{align*}
\norm{\div_\e^{-1}(f)}{\Oe} + \e\norm{\nabla \div_\e^{-1}(f)}{\Oe}\leq \norm{f}{L^2(\Oe)}
\end{align*}
for every $f \in L^2(\Oe)$.
\end{lemma}
\begin{proof}
By Lemma \ref{lemma:div-1} there exists a linear continuous operator $\div^{-1} : L^2(\Omega) \rightarrow H^1(\Omega)^N$ such that $\div(\div^{-1}(f))= f$ for every $f \in L^2(\Omega)$ 
Using this operator and the restriction operator $R_\e$ of Lemma \ref{lemma:Re}, we can define
\begin{align*}
\div_\e^{-1}( f) \coloneqq R_\e (\div^{-1} (\widetilde{f})),
\end{align*}
where $\widetilde{f}$ denotes the extension of $f$ by $0$ to $\Omega \setminus \Oe$.
	
The explicit formula for $\div \circ R_\e$ from Lemma \ref{lemma:Re} shows
\begin{align*} 
\div (\div_\e^{-1} (f)) &= \div (R_\e ( \div^{-1} (\widetilde{f}))) =
\\
&
=\div (\div^{-1} (\widetilde{f})) + \frac{1}{|\Yp|} \sum\limits_{k\in I_\e} \chi_{k+ \e\Yp}\int\limits_{k + \e\Ys} \div (\div^{-1}  (\widetilde{f}(x))) \dx
\\
&=
\widetilde{f} + \frac{1}{|\Yp|} \sum\limits_{k\in I_\e} \chi_{k+ \e\Yp}\int\limits_{k + \e\Ys} \widetilde{f}(x) \dx
= f.
\end{align*}
Moreover, using the estimate of Lemma \ref{lemma:Re}, we obtain, for $\e \leq 1$, 
\begin{align*}
\norm{ \div_\e^{-1}(f)}{\Oe} + \e\norm{ \nabla \div_\e^{-1}(f)}{\Oe} 
\leq 
C\Big( \norm{\div^{-1} (\widetilde{f})}{\Omega}+ \e \norm{\nabla \div^{-1} (\widetilde{f})}{\Omega} \Big) 
\\
\leq C \norm{\nabla \div^{-1} (\widetilde{f})}{H^1(\Omega)} \leq \norm{\widetilde{f}}{\Omega} = C\norm{f}{\Oe}.
\end{align*}
\end{proof}

\subsection{Estimates of the data}
\begin{lemma}\label{lemma:EstimatesHatData}
There exists a constant $C \in L^{p_s}(S)$ such that
\begin{align*}
&\norm{\hat{f}_\e(t)}{L^2(\Omega)} + \norm{\hat{p}_{b,\e}(t)}{L^2(\Omega)} +
\norm{\nabla\hat{p}_{b,\e}(t)}{L^2(\Omega)} 
\\
&+\frac{1}{\e}\norm{\hat{v}_{\Ge}(t)}{L^2(\Omega)} + \norm{\nabla \hat{v}_{\Ge}(t)}{L^2(\Omega)} \leq C(t)
\end{align*}
for a.e.~$t\in S$.
\end{lemma}
\begin{proof}
We note that $D \psi_\e^{-1}\tx = \Psi_\e^{-1}(t,\psi_\e^{-1}\tx)$. Hence, we get
\begin{align*}
\norm{\hat{f}_\e(t)}{\Omega} = \Big(\intO f_\e(t,\psi_\e\tx))^2 \dx\Big)^\frac{1}{2} = \Big(\intO \det(\Psi_\e^{-1}(t,\psi_\e^{-1}\tx)) f_\e\tx^2\dx\Big)^\frac{1}{2} 
\\
\Big(\intO J_\e^{-1}(t,\psi_\e^{-1}\tx) f_\e\tx^2\dx\Big)^\frac{1}{2} 
\leq C \Big(\intO f_\e\tx^2\dx\Big)^\frac{1}{2} \leq \norm{f_\e(t)}{\Omega}.
\end{align*}
Then, the uniform bound of $f_\e(t)$ given by Assumption \ref{ass:Data} implies the uniform bound of $\norm{\hat{f}(t)}{L^2(\Omega)} \leq C(t)$ for a.e.~$t\in S$.
By the same computation, we obtain
$\norm{\hat{p}_{b,\e}(t)}{L^2(\Omega)} +
\frac{1}{\e}\norm{\hat{v}_{\Ge}(t)}{L^2(\Omega)}\leq C(t)$.

In order to estimate the gradient, we use the chain rule and rewrite $\nabla \hat{p}_{b,\e}\tx = \Psi_\e^\top\tx \nabla p_{b,\e}(t,\psi_\e\tx)$. Then, the uniform estimates of Lemma \ref{lemma:EstimatesPsi} yields
\begin{align*}
&\norm{\nabla \hat{p}_{b,\e}(t)}{\Omega} =\Big(\intO J_\e^{-1}(t,\psi_\e^{-1}\tx) \Psi_\e^{-\top}(t,\psi_\e^{-1}\tx) \nabla p_{b,\e}^2 \tx\dx\Big)^\frac{1}{2} 
\\
&\leq\Big(\intO J_\e^{-1}(t,\psi_\e^{-1}\tx) \nabla p_{b,\e}\tx^2\dx\Big)^\frac{1}{2} 
\leq C \Big(\intO \nabla p_{b,\e}^2\tx\dx.\Big)^\frac{1}{2} \leq C \norm{\nabla p_{b,\e}(t)}{\Omega}.
\end{align*}
The uniform bound of $\norm{\nabla p_{b,\e}(t)}{\Omega}$ given by Assumption \ref{ass:Data} implies the uniform bound of $\norm{\nabla \hat{p}_{b,\e}(t)}{L^2(\Omega)} \leq C(t)$ for a.e.~$t\in S$.
By the same computation, we obtain
$\norm{\nabla \hat{v}_{\Ge}(t)}{L^2(\Omega)}\leq C(t)$.
\end{proof}

\begin{proof}[Proof of Theorem \ref{thm:ExistenceEps}]
First, we show that there exists a solution $(\hat{w}_\e(t),\hat{q}_\e(t)) \in H^1_\Ge(\Oe)^N \times L^2(\Oe)$ of \eqref{eq:WeakFormTransformed1}--\eqref{eq:WeakFormTransformed2} for a.e.~$t\in S$.
Due to the Poincar\'e inequality from Lemma \ref{lemma:PoincareEps}, $
\norm{\cdot}{V_\e} : \HepsN \rightarrow \R,	\ 	v \mapsto \norm{v}{V_\e} := \e \norm{\nabla v}{\Oe}$
defines a norm on $\HepsN$.
We define the following bilinear forms for a.e.\ $t\in S$ 
\begin{align}\notag
&a_{\e}(t) : \HepsN \times \HepsN \rightarrow \R, \ \ \
&&(v,w) \mapsto (\e^2 \nu A_\e(t) \hat{e}_{\e,t}(v), \nabla w)_\Oe,
\\\notag
&b_{\e}(t) : \HepsN \times L^2(\Oe) \rightarrow \R, \ \ \
&&(v,p) \mapsto (\div(A_\e(t) v), p)_\Oe. 
\end{align}
Using the uniform estimates of Lemma \ref{lemma:EstimatesPsi} and the Korn-type inequality for the two-scale transformation method (see.~Lemma \ref{lemma:KornPsiEps}), we obtain the following uniform coercivity and continuity estimate for the bilinear form $a_\e(t)$:
\begin{align}\notag
a_{\e}&(t)(w,w) = (\e^2 \nu A_\e(t) \hat{e}_{\e,t}(w), \nabla w)_\Oe  = (\e^2 \nu J_\e(t) \hat{e}_{\e,t}(w), \Psi_\e^{-\top}(t) \nabla w)_\Oe 
\\\notag
&=
(\e^2 \nu J_\e(t) \hat{e}_{\e,t}(w), (\Psi_\e^{-\top} \nabla w + (\Psi_\e^{-\top} \nabla w)^\top)/2)_\Oe
\\\notag
&= (\e^2 \nu J_\e(t) \hat{e}_{\e,t}(w),\hat{e}_{\e,t}(w))_\Oe
\\\label{eq:CoercivityAeps}
&\geq \e^2 \nu c_J 
\norm{\hat{e}_{\e,t}(w)}{\Oe}^2
\geq  \e^2 \nu c_J \alpha \norm{\nabla w}{\Oe}^2
\geq C \norm{w}{V_\e}^2,
\end{align}
\begin{align}\label{eq:ContinuityAeps}
a_{\e}(t)(v,w) = (\e^2 \nu J_\e(t) \Psi_\e^{-1}(t) (\Psi_\e^{-\top} \nabla v + (\Psi_\e^{-\top} \nabla v), \nabla w)_\Oe 
\leq C \norm{v}{V_\e} \norm{w}{V_\e}.
\end{align}
In order to give a uniform estimate of the inf--sup constant, we choose an arbitrary $\phi \in L^2(\Oe)$.
Then, Lemma \ref{lemma:dive-1} gives a constant $C \in \R$, independent of $\e$ and $\phi$, and a function $\hat{v} \in \HepsN$ such that 
\begin{align}
\div(\hat{v}) = \phi,\ \ \
\e\norm{\hat{v}}{\Oe} \leq C \norm{\phi}{\Oe}.
\end{align}
We define $v \coloneqq J_\e^{-1}(t) \Psi_\e(t) \hat{v} \in \HepsN$. Using the product rule, the estimates from Lemma \ref{lemma:EstimatesPsi} and the $\e$-scaled Poincar\'e inequality (cf.~Lemma \ref{lemma:PoincareEps}), we obtain:
\begin{align}\notag
\norm{v}{V_\e} &= \e \norm{D(J_\e^{-1}(t) \Psi_\e(t) \hat{v})}{\Oe} 
\leq
\e \norm{D(J_\e^{-1}(t) \Psi_\e(t)) \hat{v}}{\Oe} 
+
\e \norm{J_\e^{-1}(t) \Psi_\e(t) D\hat{v}}{\Oe} 
\\\notag
&\leq C \norm{\hat{v}}{\Oe} + C \e \norm{D\hat{v}}{\Oe}
\leq C \norm{\hat{v}}{V_\e}
\leq C \norm{\phi}{\Oe}.
\end{align}
With this choice of $v$, we see that
\begin{align}\label{eq:InfsupBeps}
\sup\limits_{v\in \HepsN} \frac{|b_\e(t)(v, \phi)|}{|v|_{V_\e}} 
\geq
\frac{(\phi,\phi)_{\Oe}}{\norm{v}{V_\e}}
\geq
\frac{\norm{\phi}{\Oe}^2}{C\norm{\phi}{\Oe}}
= C\norm{\phi}{\Oe}
\end{align}
for a.e.\ fixed $t \in S$.

The continuity of the bilinear form $b_\e(t)$ follows with the product rule and the $\e$-scaled Poincar\'e inequality:
\begin{align}\notag
b_\e(t)(v,p) &=
(\div(A_\e(t) v),p)
\leq
(\norm{\div(A_\e(t)) v}{\Oe}		
+
\norm{A_\e(t) : \nabla v}{\Oe} ) \norm{p}{\Leps}
\\\label{eq:ContinuityBeps}
&\leq C(\e^{-1} \norm{v}{\Oe} + \norm{\nabla v}{\Oe}) \norm{p}{\Leps}
\leq \e^{-1}C\norm{v}{V_\e} \norm{p}{\Leps}.
\end{align}
Note that a more precise estimate like for the inf--sup constant does not yield an $\e$-independent constant for $b_\e(t)$. However, the norm of $b_\e$ does not appear in the right-hand sides of \eqref{eq:estimateV} and \eqref{eq:estimateP}. Nevertheless, the continuity and the coercivity constant of $a_\e(t)$ as well as the inf--sup constant of $b_\e(t)$, which occur in \eqref{eq:estimateV} and \eqref{eq:estimateP}, do not depend on $\e$ or $t$. 
	
Now, we estimate the right-hand sides of \eqref{eq:WeakFormTransformed1} and \eqref{eq:WeakFormTransformed2}.
For the first summand of the right-hand side of \eqref{eq:WeakFormTransformed1}, we obtain with Lemma \ref{lemma:EstimatesPsi}, Lemma \ref{lemma:EstimatesHatData} and the $\e$-scaled Poincar\'e inequality
\begin{align}\notag
&\norm{J_\e(t)\hat{f}_\e(t)- A_\e(t) \nabla \hat{p}_{b,\e}}{V_\e'}  
\\\notag
&=
\sup\limits_{v \in \HepsN} 
\frac{\intOe (J_\e\tx\hat{f}_\e\tx- A^\top_\e\tx \nabla \hat{p}_{b,\e}\tx) \cdot v(x) \ \dx}{\norm{v}{V_\e}}
\\\label{eq:Estf1}
&\leq
\frac{C(t) \norm{v}{\Oe}}{\norm{v}{V_\e}}
\leq
\frac{C(t) \norm{v}{V_\e} }{\norm{v}{V_\e}}
\leq
C(t),
\end{align}
where $C \in L^{p_s}(S)$.
We rewrite the second summand of \eqref{eq:WeakFormTransformed1} and obtain from the continuity estimate \eqref{eq:ContinuityAeps} of $a_\e(t)$ 
\begin{align}\label{eq:Estf2}
\norm{-a_\e(t)(\partial_t\psi(t), \cdot )}{V_\e'} \leq
C \norm{\partial_t \psi(t)}{V_\e'} \leq \e C.
\end{align}
Later this term will also vanish during the homogenisation because it is of order $O(\e)$.
We can estimate the right-hand side of \eqref{eq:WeakFormTransformed2}  with the continuity estimate \eqref{eq:ContinuityBeps} of $b_\e(t)$ and Lemma \ref{lemma:EstimatesHatData} by
\begin{align}\label{eq:Estg}
\norm{-b_\e(t)(\hat{v}_\Ge(t), \cdot)}{L^2(\Oe)'} \leq  \e^{-1} C \norm{\hat{v}_\Ge(t)}{\Oe}\leq \e^{-1}\e C(t) \leq C(t)
\end{align}
for $C \in L^p(S)$.
Using  Lemma~\ref{lemma:SaddlePoint} with the estimates \eqref{eq:CoercivityAeps}, \eqref{eq:ContinuityAeps}, \eqref{eq:InfsupBeps}, \eqref{eq:Estf1}, \eqref{eq:Estf2} and \eqref{eq:Estg} yields, for $\e>0$ small enough and a.e.~$t \in S$, the existence of a unique solution $(\hat{w}_\e(t), \hat{q}_\e(t)) \in \HepsN \times L^2(\Oe)$ of \eqref{eq:WeakFormTransformed1}--\eqref{eq:WeakFormTransformed2} such that
\begin{align}
\norm{\hat{w}_\e(t)}{V_\e} + \norm{\hat{q}_\e}{\Oe} \leq C(t)
\end{align}
for $C \in L^p(S)$.
	
By the definition of the norm $V_\e$ and Lemma \ref{lemma:PoincareEps}, we can estimate further
\begin{align}
\norm{\hat{w}_\e(t)}{\Oe} + \e \norm{\nabla \hat{w}_\e(t)}{\Oe} + \norm{\hat{q}_\e(t)}{\Oe} \leq C(t)
\end{align}
for $\e>0$ small enough, a.e.~$t \in S$ and $C \in L^p(S)$.
	
By Lemma \ref{lemma:ContinuityWithRespectToData}, we get the continuity of the solution with respect to the bilinear forms and the right-hand sides. Since the bilinear forms and the right-hand sides are measurable in time, $\hat{w}_\e: S \rightarrow \HepsN$ and $\hat{q}_\e : S \rightarrow \Leps$ are measurable functions and thus $\hat{w}_\e \in L^{p_s}(S,\HepsN)$ and $\hat{q}_\e \in L^{p_s}(S;\Leps)$
\end{proof}

\begin{remark}
Another ansatz for the proof of Theorem \ref{thm:ExistenceEps} would be to substitute $\hat{w}_\e$, so that we obtain a homogeneous divergence condition. Then, we could use the Lemma of Lax--Milgram for functions $v$ with $\div(A_\e(t) v )= 0$. Using the same preliminary work, we could prove the same uniform bounds for the velocity and the pressure. But the difficulty with this ansatz is that the measurability with respect to time cannot be concluded directly, since the space of functions $v$ satisfying $\div(A_\e(t) v)= 0$  depends on time.
\end{remark}

\section{Homogenisation in the periodic reference domain }\label{section:Homogenisation}
In this section, we pass to the limit $\e \to 0$ in \eqref{eq:WeakFormTransformed1}--\eqref{eq:WeakFormTransformed2} using the notion of two-scale convergence (cf.~\cite{All92}, \cite{LNW02}) and derive the two-pressure Stokes system \eqref{eq:WeakLimitTransformed1}--\eqref{eq:WeakLimitTransformed3} as two-scale limit problem. 

Find $(\hat{w}_0, \hat{q},\hat{q}_1) \in L^{p_s}(S;\HYperN) \times L^{p_s}(S;H^1_0(\Omega)) \times L^{p_s}(S;L^2(\Omega;L^2_0(\Yp)))$ such that
\begin{align}\notag
\intOYp \nu A_0\txy \Psi_0^{-\top}\txy\nabla_y \hat{w}_0\txy :   \nabla_y\varphi\xy \dyx
\\\notag
+\intOYp A_0^{\top}\txy\nabla_x \hat{q}\tx  \cdot \varphi\xy + \hat{q}_1(t)\div_y(A_0\txy\varphi(x,y)) \dyx
\\\label{eq:WeakLimitTransformed1}
= \intOYp (J_0\txy f\tx - A_0^\top\txy\nabla p_b\tx) \cdot \varphi(x,y) \dyx
\\\label{eq:WeakLimitTransformed2}
\intOYp\div_y (A_0\txy\hat{w}_0\txy )  \theta_1(x,y) \dyx  = 0
\\\notag
\int\limits_{\Omega}   \div_x \Big(\intYp A_0\txy\hat{w}_0\txy   \dy \Big) \ \theta_0(x)  \dx  
\\\label{eq:WeakLimitTransformed3}
=
-\int\limits_{\Omega}   \intYp \div_y \big(A_0\txy \hat{v}_\Gamma\txy \big)  \dy  \ \theta_0(x)  \dx 
\end{align}
for every
$(\varphi, \theta_0,\theta_1) \in \HYperN \times H^1_0(\Omega) \times L^2(\Omega;L^2_0(\Yp))$.

\begin{defi}
Let $p,q \in (1,\infty)$ with $\frac{1}{p} + \frac{1}{q}=1$. We say that a sequence $u_\e$ two-scale converges weakly to $u_0 \in L^p(\Omega \times Y)$ if
\begin{align*}
\lim\limits_{\e\to 0}\intO u_\e(x) \varphi \xxeps dx = \intOY u_0\xy \varphi\xy dydx
\end{align*}
for every $\varphi \in L^q(\Omega; C_\#(Y))$.
If additionally  $		\limEpsNull \norm{v_\e}{L^p(\Omega)} = \norm{v_0}{L^p(\Omega \times Y)}$, we say $u_\e$ two-scale converges strongly to $u_0$.
\end{defi}
The following theorem is one of the fundamental compactness results in the notion of two-scale convergence.
\begin{theorem}\label{thm:T-SC-LpComp}
Let $p \in (1,\infty)$ and let $u_\e$ be a bounded sequence in $L^p(\Omega)$. Then, there exists a subsequence and $u_0 \in L^p(\Omega \times Y)$ such that this subsequence two-scale converges weakly to $u_0$.
\end{theorem}
The following result allows us to handle the coefficients in the homogenisation.
\begin{lemma}
Let $1 <p,q,r< \infty$ with $\frac{1}{p}+ \frac{1}{q} = \frac{1}{r}$. Let $u_\e$ be a sequence in $L^p(\Omega)$ which two-scale converges strongly to $u_0 \in L^p(\Omega \times Y)$ and let $v_\e$ be a sequence in $L^q(\Omega)$ which  two-scale converges weakly (resp.~strongly) to $v_0 \in L^q(\Omega \times Y)$. Then, $u_\e v_\e$ is a sequence of functions in $L^r(\Omega)$ which two-scale converges weakly (resp.~strongly) to $u_0 v_0 \in L^r(\Omega \times Y)$.
\end{lemma}

In order to translate the two-scale convergence of the data and the solution between the actual and the transformed configuration, we use the following transformation result of \cite{Wie21}.
\begin{theorem}\label{thm:Tsc-Trafo}
Let $p \in (1,\infty)$ and $t\in S$. Let $u_\e$ be a sequence in $L^p(\Omega)$ and $\hat{u}_\e = u_\e \circ \psi_\e(t)$ with $\psi_\e$ as in Assumption \ref{ass:psi}. Then, for a.e.~$t\in S$, $u_\e$ two-scale converges weakly with respect to the $L^p$-norm to $u_0 \in L^p(\Omega \times Y)$ if and only if $\hat{u}_\e$ two-scale converges weakly with respect to the $L^p$-norm to $\hat{u}_0 \in L^p(\Omega \times Y)$. Moreover, $\hat{u}_0(\cdot_x, \cdot_y)= u_0(\cdot_x, \psi_0(\cdot_x\cdot_y))$ holds and, equivalently, $u_0 = \hat{u}_0(\cdot_x, \psi_0^{-1}(t,\cdot_x\cdot_y))$.
\end{theorem}
\begin{proof}
For a proof of Theorem \ref{thm:Tsc-Trafo} see \cite{Wie21}. Note, that there the deformations $\psi_\e$ are only defined on $\Omega_\e \subset \Omega$, which is why the transformation result has to deal with the extension of the functions by $0$. However, the results there holds for $\Omega_\e = \Omega$, which proves Theorem \ref{thm:Tsc-Trafo}.
\end{proof}
Further two-scale transformation results for weakly differentiable functions can be found in \cite{Wie21}.
Moreover, the following result of \cite{Wie21} allows to transforms the strong two-scale convergence.
\begin{theorem}\label{thm:Tsc-TrafoStrong}
Let $p \in (1,\infty)$ and $t \in S$. Let $u_\e$ be a sequence in $L^p(\Omega)$ and $u_0\in L^p(\Omega\times Y))$ such that $u_\e$ two-scale converges weakly to $u_0$. 
Let $\hat{u}_\e = u_\e \circ \psi_\e(t)$ be a sequence in $L^p(\Omega)$ with $\psi_\e$ as in Assumption \ref{ass:psi} such that $\hat{u}_\e$ two-scale converges weakly to $\hat{u}_0 \in L^p(\Omega\times Y)$ with $\hat{u}_0 = u_{0}(\cdot_x, \psi_0(t,\cdot_x,\cdot_y)$. Then, for a.e.~$t\in S$, the following statements hold:
\begin{enumerate}
\item If $u_\e$ two-scale converges strongly to $u_0$ with respect to the $L^p$-norm, then $\hat{u}_\e$ two-scale converges strongly to $\hat{u}_0$ with respect to every $L^{p'}$-norm for $p' \in (1,p)$,
\item If $\hat{u}_\e$ two-scale converges strongly to $\hat{u}_0$ with respect to the $L^p$-norm, then $u_\e$ two-scale converges strongly to $u_0$ with respect to every $L^{p'}$-norm for $p' \in (1,p)$.
\end{enumerate}
\end{theorem}
\begin{proof}
The proof of Theorem \ref{thm:Tsc-TrafoStrong} follows in the same way as Theorem \ref{thm:Tsc-Trafo} from the results of \cite{Wie21}.
\end{proof}

\subsection{Two-scale convergence of the transformed data}
\begin{lemma}
Let $\hat{f}_\e, \hat{p}_{b,\e}$ and $\hat{v}_\Ge$ be defined by \eqref{eq:TrafoData}. Then, 
\begin{align*}
&\hat{f}_\e(t) \to f(t) &&\textrm{weakly in the two-scale sense},
\\
&\nabla \hat{p}_{b,\e}(t) \to \nabla_x p_{b}(t) + \nabla_y \hat{p}_{b,1}  &&\textrm{weakly in the two-scale sense},
\\
&\frac{1}{\e}\hat{v}_{\Ge}(t) \to \hat{v}_{\Gamma}(t)  &&\textrm{weakly in the two-scale sense},
\\
&\nabla \hat{v}_{\Ge}(t) \to \nabla_y \hat{v}_{\Gamma}(t)  &&\textrm{weakly in the two-scale sense}
\end{align*}
for a.e.~$t\in S$, where $f$, $p_b$ and $v_\Gamma$ are the two-scale limits given in Assumption~\ref{ass:Data}, $\hat{p}_{b,1}(t) = p_{b,1}(t, \cdot_x, \psi_0(t, \cdot_x, \cdot_y)) + \check{\psi}_0 \cdot \nabla_x p_{b}(t)$ and
$\hat{v}_\Gamma(t) = v_\Gamma(t, \cdot_x, \psi_0(t, \cdot_x, \cdot_y))$.
\end{lemma}
\begin{proof}
The two-scale convergence of $\hat{f}_\e(t)$ and of $\frac{1}{\e} \hat{v}_{\Ge}$ follows from Theorem \ref{thm:Tsc-Trafo}. The two-scale convergence of $\nabla \hat{p}_{b,\e}(t)$ follows from \cite[Theorem 3.10]{Wie21} and the the two-scale convergence of $\nabla \hat{v}$ from \cite[Theorem 3.9]{Wie21}.
\end{proof}
Moreover, we need the two-scale convergence of the transformation coefficients, which are given by the following result.
\begin{lemma}\label{lemma:TscPsi}
Let $\psi_\e$ and $\psi_0$ be given by Assumption \ref{ass:psi}. Then,
\begin{align*}
&\Psi_\e(t) \to \Psi_0(t), \
\Psi_\e^{-1}(t) \to \Psi_0^{-1}(t), \ J_\e(t) \to J_0(t), J_\e^{-1}(t) \to J_0^{-1}(t), A_\e(t) \to A_0(t),
\\
&\e D_x A_\e(t) \to D_y A_0(t), \ \e D_x A_\e^{-1}(t) \to D_y A_0(t)
\end{align*}
strongly in the two-scale sense, for a.e.~$t\in S$, where the strong two-scale convergence holds with respect to every $L^p$-norm for $p  \in (1,\infty)$ and $A_0 \coloneqq J_0 \Psi_0^{-1}$.
\end{lemma}
\begin{proof}
Using the results of \cite{Wie21} it remains to prove that $\e D_x A_\e(t) \to D_y A_0(t)$ and $\e D_x A_\e^{-1}(t) \to D_y A_0(t)$ strongly in the two-scale sense.
Therefore, we note that $\e D_x \Psi_\e(t) = \e D_x \nabla \psi_\e^\top(t) $ two-scale converges strongly to $D_y \Psi_0(t) = D_y \nabla_y \psi_0^\top(t)$ by Assumption \ref{ass:psi}.

We rewrite $D_x A_\e(t)$ into the sum of polynomials like in the proof of Lemma \ref{lemma:EstimatesPsi}. Then, we pass to the limit $\e \to 0$ using the two-scale convergence of $\e D_x \Psi_\e(t)$ and $\Psi_\e(t)$.

By the same argumentation, we obtain the strong two-scale convergence of $\e D_x J_\e(t)$ to $D_y J_0(t)$.
Then, we rewrite $A_\e^{-1}(t) = J_\e^{-1}(t) \Psi_\e(t)$. Using the quotient rule and the previously proven two-scale convergences, we obtain the strong two-scale convergence of $\e D_x A_\e^{-1}(t)$ to $D_y A_0^{-1}(t)$.
\color{black}
\end{proof}

\subsection{Homogenisation of the transformed Stokes equations}
\begin{theorem}\label{thm:HomogenisationTransformed}
Let $\hat{w}_\e$ and $\hat{q}_\e(t)$ be the solution of \eqref{eq:WeakFormTransformed1}--\eqref{eq:WeakFormTransformed2}. Let $\hat{Q}_\e$ be the extension of $\hat{q}_\e$ as defined in Lemma \ref{lemma:ConvergenceQe} and $\widetilde{\hat{w}_\e} \in L^{p_s}(S;H^1(\Omega))$ be the extension of $\hat{w}_\e$ by $0$ on $\Omega \setminus \Oe$.
Then, $\widetilde{\hat{w}_\e}(t)$ two-scale converges to $\hat{w}_0(t)$  and $\hat{Q}_\e(t)$ converges strongly in $L^2(\Omega)$ to $\hat{q}(t)$, for a.e.~$t\in S$, where $(\hat{w}_0, \hat{q}, \hat{q}_1) \in L^{p_s}(S;\HYperN) \times L^{p_s}(S;H^1_0(\Omega)) \times L^{p_s}(S;L^2(\Omega;L^2_0(\Yp)))$ is the unique solution of \eqref{eq:WeakLimitTransformed1}--\eqref{eq:WeakLimitTransformed3}.
\end{theorem}

In order to pass to the limit $\e \to 0$ in \eqref{eq:WeakFormTransformed1}, we test it by $A_\e^{-1}(t)\varphi$ and obtain
\begin{align}\notag
&( \e^2 \nu  A_ \e(t) \hat{e}_{ \e,t}(\hat{w}_ \e(t)), \nabla ( A_ \e^{-1}(t) \varphi))_\Oe 
- (\hat{Q}_ \e(t), \div(\varphi))_\Oe
\\\label{eq:WeakFormSub}
&
= 
(\Psi_\e^{\top}(t) \hat{f}_ \e(t)-  \nabla{\hat{p}}_{b, \e}, \varphi)_\Oe
-
( \e^2 \nu  A_ \e(t) \hat{e}_{ \e,t}(\hat{v}_\Ge(t)), \nabla( A_ \e^{-1}(t) \varphi))_{\Omega_ \e}.
\end{align}
Since $A_\e^{-1}(t)$ is invertible \eqref{eq:WeakFormTransformed1} can be replaced by \eqref{eq:WeakFormSub}.

First, we prove the strong convergence of $\hat{Q}_\e(t)$. Thereto, we transfer the argumentation of \cite{All89} on our weak form with the different function spaces.
\begin{lemma}\label{lemma:ConvergenceQe}
Let $\hat{q}_\e$ be the second part of the solution of \eqref{eq:WeakFormTransformed1}--\eqref{eq:WeakFormTransformed2} and $\hat{Q}_\e$ be the extension of $\hat{q}_\e$ on $\Omega$ defined by
\begin{align}\label{eq:ExtensionQe}
\hat{Q}_\e\tx 
\coloneqq 
\begin{cases}
\hat{q}_\e\tx & \textrm{ if } x \in \Oe,
\\
\frac{1}{|\Yp|} \int\limits_{k + \e \Yp} \hat{q}_\e\tx & \textrm{ if } x \in k + \e \Yp \textrm{ for } k \in I_\e.
\end{cases}
\end{align}
Then, for a.e.~$t\in S$, there exists  $\hat{q}(t) \in L^2(\Omega)$ and a subsequence of $\hat{Q}_\e(t)$ which converges strongly in $L^2(\Omega)$ to $q(t)$.
\end{lemma}

\begin{proof}
We define $F_\e(t) \in (H^1(\Omega)^N)'$ by
\begin{align}
\langle F_\e(t), \varphi \rangle_{H^1(\Omega)', H^1(\Omega)}
\coloneqq \intOe \hat{q}_\e\tx \div(R_\e \varphi(x)) \dx.
\end{align}
From \eqref{eq:WeakFormSub}, we obtain
\begin{align*}
\intO \hat{q}_\e(t) \div(R_\e \varphi) \dx = (\e^2 \nu  A_ \e(t) \hat{e}_{\e,t}(\hat{w}_ \e(t)), \nabla ( A_ \e^{-1}(t) R_\e \varphi))_\Oe 
\\
- 
(\Psi_\e^{\top}(t) \hat{f}_ \e(t)-  \nabla{\hat{p}}_{b, \e}, R_\e\varphi)_\Oe
-
( \e^2 \nu  A_ \e(t) \hat{e}_{ \e,t}(\hat{v}_\Ge(t)), \nabla( A_ \e^{-1}(t) R_\e \varphi))_{\Omega_ \e}.
\end{align*}
Thus, we can estimate $F_\e(t)$ using the estimates of $\e \nabla \hat{w}_\e(t)$ (cf.~\eqref{eq:EpsIndependentEstimateVP}), the coefficients (cf.~Lemma \ref{lemma:EstimatesPsi}) and the data (cf.~Assumption~\ref{ass:Data}) as well as the product rule by
\begin{align*}
|\langle F_\e(t), \varphi \rangle_{H^1(\Omega)', H^1(\Omega)} | &\leq  C \e\norm{\nabla (A_\e^{-1} R_\e\varphi)}{\Oe} + C \norm{R_\e \varphi}{\Oe} + \e^2 \norm{\nabla (A_\e^{-1} R_\e\varphi)}{\Oe}
\\
&\leq
C (\e + \e^2)\norm{\nabla R_\e\varphi}{\Oe} + C (1+\e)\norm{R_\e \varphi}{\Oe}.
\end{align*}
Then, the estimates of Lemma \ref{lemma:Re} imply, for $\e\leq 1$,
\begin{align}\label{eq:EstimateFe}
|\langle F_\e(t), \varphi \rangle_{H^1(\Omega)', H^1(\Omega)} |\leq  C (\norm{\varphi}{\Omega}+ \e\norm{\nabla \varphi}{\Omega})
\end{align}
and in particular $\norm{F_\e(t)}{H^1(\Omega)'} \leq C$.

Because $\div(R_\e\varphi) = 0$ if $\div(\varphi)= 0$, we obtain
\begin{align*}
\langle F_\e(t), \varphi \rangle_{H^1(\Omega)', H^1(\Omega)} 
= \intOe \hat{q}_\e\tx \div(R_\e \varphi(x)) \dx = \intOe \hat{q}_\e\tx \div(\varphi(x)) \dx= 0
\end{align*}
for every $\varphi \in H^1(\Omega)$ with $\div(\varphi) =0$.
Since $\div$ has closed range (cf.~Lemma \ref{lemma:div-1}), the closed-range theorem implies that there exists $\hat{Q}_\e(t) \in L^2(\Omega)$ such that
\begin{align}\label{eq:NablaQe=Fe}
&\intO \hat{Q}_\e(t) \div(\varphi) \dx = \langle F_\e(t), \varphi \rangle_{H^1(\Omega)', H^1(\Omega)} 
= \intOe \hat{q}_\e\tx \div(R_\e \varphi(x)) \dx.
\end{align}
Moreover, we obtain with Lemma \ref{lemma:dive-1} the uniform boundedness of  $\norm{\hat{Q}_\e(t)}{L^2(\Omega)}$ by
\begin{align}\notag
\norm{\hat{Q}_\e(t)}{L^2(\Omega)}^2 & = \intO \hat{Q}_\e(t) \div(\div^{-1}(\hat{Q}_\e(t))) = |\langle F_\e(t),\div^{-1}(\hat{Q}_\e(t))  \rangle_{H^1(\Omega)', H^1(\Omega)}| 
\\\label{eq:HatQeBounded}
&\leq C \norm{\div^{-1}(\hat{Q}_\e(t))}{H^1(\Omega)} \leq C \norm{\hat{Q}_\e(t)}{L^2(\Omega)} .
\end{align}
In order to identify $\hat{Q}_\e(t)$ with $\hat{q}_\e$ on $\Oe$, we note that $R_\e(\widetilde{\varphi}) = \varphi$ for every $\varphi \in H^1_\Ge(\Oe)$, where $\widetilde{\varphi}$ is the extension by $0$ of $\varphi$. Then, we obtain
\begin{align*}
\intO \hat{Q}_\e(t) \div(\widetilde{\varphi}) \dx 
=
\intO \hat{Q}_\e(t) \div(\varphi) \dx 
= \intOe \hat{q}_\e(t) \div(\varphi) \dx.
\end{align*}
Lemma \ref{lemma:dive-1} gives the existence of a function $\varphi \in H^1_\Ge(\Oe)$ with $\div(\varphi) = \hat{Q}_\e(t) - \hat{q}_\e(t)$. Testing with this $\varphi$ implies  $\hat{Q}_\e(t)=\hat{q}_\e(t)$ on $\Oe$.

In order to show the strong convergence of $\hat{Q}_\e(t)$, we note that the boundedness of $\hat{Q}_\e(t)$ in $L^p(\Omega)$ allows us to pass to a subsequence, which we still denote by $\hat{Q}_\e(t)$, such that $\hat{Q}_\e(t)$ converges weakly to a function $\hat{q}(t) \in L^2(\Omega)$. 
Now, we consider $\varphi_\e = \div^{-1}(\hat{Q}_\e(t))$, which converges weakly to $\varphi =\div^{-1}(\hat{q}(t))$ in $H^1(\Omega)$, where $\div^{-1}$ is given by Lemma \ref{lemma:div-1}.
Then, we obtain from \eqref{eq:NablaQe=Fe} and \eqref{eq:EstimateFe}
\begin{align*}
(\hat{Q}_\e(t), \div(\varphi_\e - \varphi))_\Oe \leq C( \norm{\varphi_\e -\varphi}{\Omega} + \e \norm{\nabla (\varphi_\e -\varphi)}{\Omega}).
\end{align*}
The compact embedding of $H^1(\Omega)$ into $L^2(\Omega)$ implies that $\norm{\varphi_\e -\varphi}{\Omega} \to 0$. Since $\norm{\nabla (\varphi_\e -\varphi)}{\Omega}$ is bounded, we obtain 
\begin{align*}
(\hat{Q}_\e(t), \hat{Q}_\e(t) -\hat{q}(t)) = (\hat{Q}_\e(t), \div(\varphi_\e - \varphi))_\Oe \to 0.
\end{align*}
By using additionally the weak convergence of $\hat{Q}_\e(t)$, we obtain
\begin{align*}
\norm{\hat{Q}_\e(t) - \hat{q}(t)}{\Omega}^2 
\leq (\hat{Q}_\e(t), \hat{Q}_\e(t) -\hat{q}_\e(t)) - (\hat{q}(t), \hat{Q}_\e(t) -\hat{q}_\e(t)) \to 0,
\end{align*}
which shows the strong convergence of $\hat{Q}_\e(t)$.

The explicit formula \eqref{eq:ExtensionQe} of $\hat{Q}_\e(t)$ can be directly transfered from \cite{All89}.
\end{proof}

In the second step, we pass to the limit $\e \to 0$ in the divergence condition \eqref{eq:WeakFormTransformed2} and derive the  microscopic incompressibility condition \eqref{eq:WeakLimitTransformed2} and macroscopic compressibility condition \eqref{eq:WeakLimitTransformed3}.
\begin{lemma}\label{lemma:HomogenisierungDivergence}
Let $\hat{w}_\e \in L^{p_s}(S;\Oe)$ be the first part of the solution of \eqref{eq:WeakFormTransformed1}--\eqref{eq:WeakFormTransformed2} and $\widetilde{\hat{w}_\e}$ defined as in Theorem \ref{thm:HomogenisationTransformed}.
Then, there exists, for a.e.~$t\in S$, a subsequence $\widetilde{w_\e}(t)$ and $\hat{w}_0(t)\in H^1_{\Gamma\#}(\Yp)$ such that, for this subsequence, $\widetilde{\hat{w}_\e}(t)$ and $\e \nabla \widetilde{\hat{w}_\e}(t)$ two-scale converge to $\hat{w}_0(t)$ and $\nabla_y \hat{w}_0(t)$, respectively.
Furthermore, $\hat{w}_0(t)$ satisfies  \eqref{eq:WeakLimitTransformed2} and \eqref{eq:WeakLimitTransformed3}.
\end{lemma}
\begin{proof}
The uniform estimate \eqref{eq:EpsIndependentEstimateVP} implies that 
$\widetilde{\hat{w}_\e}(
t)$ and $\e \nabla \widetilde{\hat{w}_\e}(t)$ are bounded as well. Then, by a standard two-scale compactness result there exists, for a.e.~$t \in S$, a subsequence and a function $\hat{w}_0(t) \in L^2 (\Omega;H^1_\#(Y))^N$ such that for this subsequence
$\widetilde{\hat{w}_\e}(t) \to w_0(t)$ and $\e \nabla\widetilde{ \hat{w}_\e}(t) \to \nabla_y w_0(t)$ in the two-scale sense.
Using arbitrary two-scale test functions $\varphi \in C(\overline{\Omega};C^\infty_\#(Y))^N$ which are $0$ on $\Omega \times Y^\mathrm{p}$ shows $w_0(t) = 0 $ in $\Omega \times Y^s$, which means $w_0(t) \in H^1_{\Gamma\#}(\Yp)$.
	
By applying the estimates of Lemma \ref{lemma:EstimatesPsi} and Lemma \ref{lemma:EstimatesHatData} on the incompressibility condition \eqref{eq:WeakFormTransformed2}, we obtain 
\begin{align*}
\norm{\div (A_\e(t)\hat{w}_\e(t))}{L^2(\Oe)} = \norm{\div (A_\e(t)\hat{v}_\Ge(t))}{L^2(\Oe))} \leq C(t)
\end{align*}
for a.e.~$t \in S$ for $C \in L^{p_s}(S)$.
Using this estimate, the $Y$-periodicity of $\hat{w}_0$, $A_0(t)$ and the two-scale test functions as well as the strong two-scale convergence of the coefficients given by Lemma \ref{lemma:TscPsi}, we can conclude for $\theta \in D(\Omega; C^\infty_{\#}(Y))$  and for a.e.~$t \in S$:
\begin{align*}
&\intOYp\div_y (A_0\txy\hat{w}_0\txy )  \theta(x,y) \dyx
\\
&=-\intOYp A_0\txy\hat{w}_0\txy \cdot  \nabla_y \theta(x,y) \dyx
\\
&= 
-\limEpsNull \int\limits_{\Oe}   A_\e\tx \hat{w}_\e\tx  \cdot   \bigg( \e \nabla_1 \theta \xxeps + \nabla_y  \theta \xxeps \bigg) \dx
\\
&
=
\limEpsNull \int\limits_{\Oe} \e \div (A_\e\tx\hat{w}_\e\tx)   \theta \xxeps \dx 
= 0.
\end{align*}
By the density of $D(\Omega;C_\#^\infty(Y))$ in $L^2(\Omega; L^2(Y^\mathrm{p}))$, we obtain the micropscopic incompressibility condition \eqref{eq:WeakLimitTransformed2}.
	
In order to derive the macroscopic compressibility condition, we test \eqref{eq:WeakFormTransformed2} with $\theta \in D(\Omega)$ and pass to the limit $\e\to 0$. Using the two-scale convergence of $A_\e(t)$ and $\hat{v}_\Ge(t)$ and their derivatives yields
\begin{align*}
&\int\limits_{\Omega}   \div_x \Big(\intYp A_0\txy\hat{w}_0\txy   \dy \Big) \ \theta(x)  \dx 
\\
&=
- \int\limits_{\Omega}   \intYp A_0\txy\hat{w}_0\txy   \dy \cdot  \nabla_x \theta(x)  \dx 
\\
&= 
-\limEpsNull \intOe A_\e\txy \hat{w}_\e\tx  \cdot  \nabla_x\theta(x) \dx
= 
\limEpsNull \intOe  \div_x (A_\e\tx\hat{w}_\e\tx ) \theta(x) \dx
\\
&= 
-\limEpsNull \intOe  \div_x (A_\e\tx \hat{v}_\Ge \tx )   \theta(x) \dx
\\
&= 
-\int\limits_{\Omega}   \intYp \div_y \big(A_0\txy \hat{v}_\Gamma\txy \big)  \dy  \ \theta(x)  \dx.
\end{align*}
By the density of $D(\Omega)$ in $H^1_0(\Omega)$, we can conclude the macroscopic compressibility condition \eqref{eq:WeakLimitTransformed3}.
\end{proof}
In the third step, we can pass to the limit $\e \to 0$ in \eqref{eq:WeakFormSub}.
\begin{proof}[Proof of Theroem \ref{thm:HomogenisationTransformed}]

Employing Lemma \ref{lemma:HomogenisierungDivergence} and Lemma \ref{lemma:ConvergenceQe}, there exists, for a.e.~$t\in S$, a subsequence, $\hat{q}(t) \in L^2(\Omega)$  and $\hat{w}(t) \in \HYperN$ such that, for this subsequence, $\hat{Q}_\e(t)$ converges strongly to $\hat{q}(t)$, $\widetilde{w_\e}(t)$ and $\e \nabla \widetilde{w_\e}(t)$ two-scale converge weakly to $\hat{w}_0(t)$ and $\nabla_y \hat{w}_0(t)$, respectively. We consider this subsequence in the following.

Let $\varphi \in C^\infty (\overline{\Omega}; C_{\Gamma\#}^\infty (\Yp))^N$ such that $\div_y(\varphi) = 0$. Then, we test \eqref{eq:WeakFormSub} with $\varphi(\cdot_x, \frac{\cdot_x}{\e})$ and pass to the limit $ \e \to 0$, which yields
\begin{align}\notag
\intOYp \nu A_0\txy \hat{e}_{y,t}(\hat{w}_0\txy) :   \nabla_y(A_0^{-1}\txy \varphi(x,y)) \dyx
\\\notag
-
\intOYp \hat{q}\tx  \div_x(\varphi(x,y))
= \intOYp \Psi_0^\top\txy f\tx \cdot \varphi(x,y) \dyx
\\\label{eq:HomogenizedWithP0}
-
\intOYp (\nabla_x p_b\tx+\nabla_y \hat{p}_{b,1}(t)) \cdot \varphi(x,y) \dyx
\end{align}
for any $\varphi \in C^\infty (\overline{\Omega} ; C_{\Gamma\#}^\infty (Y ))^N$ such that $\div_y(\varphi) = 0$. 
Since $C^\infty (\overline{\Omega} ; C_{\Gamma\#}^\infty (\Yp ))^N $ is dense in $\HYperN$ and $\div_y : \HYperN \rightarrow L^2(\Omega\times \Yp)$ is continuous, \eqref{eq:HomogenizedWithP0} holds for every $\varphi \in \HYperN$ with $\div_y(\varphi)= 0$.

Moreover, the following integration by parts shows that we can omit $\nabla_y \hat{p}_{b,1}(t)$
\begin{align*}
\intOYp \nabla_y \check{\psi}_0^{-1} \nabla p_{b}(t)\cdot \varphi(x,y) \dyx = -\intOYp \nabla_y \hat{p}_{b,1}(t) \div_y( \varphi(x,y)) \dyx = 0.
\end{align*}
The boundary term in this integration by parts vanish since $p_{b,1}(t)$ is $Y$-periodic.
Now, we choose $\varphi = \phi \phi_i$ in \eqref{eq:HomogenizedWithP0}, for $\phi \in C^\infty(\overline{\Omega})$ and $\phi_i \in H^1_{\Gamma\#}(\Yp)^N$, with $\phi_i =  0$ in $\Ys$, $\div(\phi_i) = 0$ and $\intYp \phi_i(y) =e_i$ for $i \in \{1, \dots, N\}$ (for the existence of $\phi_i$ see \cite[Lemma 6]{Mir16}), and obtain
\begin{align*}
\intO  -\hat{q}\tx \partial_{x_i} \phi(x) - G_i\tx \varphi(x) \dx = 0
\end{align*}
for
\begin{align*}
G\tx = &
-
\intYp \nu A_0\txy \hat{e}_{y,t}(\hat{w}_0\txy) :   \nabla_y(A_0^{-1}\txy \varphi(x,y)) \dy
\\
&+
\intOYp (\Psi_0^\top\txy f\tx - \nabla_x p_b\tx) \cdot \varphi(x,y) \dy.
\end{align*}
Since $G_i(t) \in L^2(\Omega)$ it follows that $\hat{q}(t) \in H^1_0(\Omega)$.	
	
Employing the Bogovski\v{i}-operator on $\Yp$, we obtain that $\div : \HYperN \supset L^2(\Omega;H^1_0(\Yp)) \rightarrow L^2(\Omega;L^2_0(\Yp))$ is surjective. Then, the closed-range theorem gives $\hat{q}_1(t) \in L^2(\Omega; L^2_0(\Yp))$ such that
\begin{align}\notag
\intOYp \nu A_0\txy \hat{e}_{y,t}(\hat{w}_0\txy) :   \nabla_y(A_0^{-1}\txy \varphi(x,y)) \dyx
\\\notag
+\intOYp \nabla_x \hat{q}_0\tx  \cdot \varphi(x,y) + \hat{q}_1(t)\div_y(\varphi(x,y)) \dyx
\\\label{eq:HomogenizedWithP0P1}
= \intOYp (\Psi_0^\top\txy f\tx - \nabla p_b\tx) \cdot \varphi(x,y) \dyx
\end{align}
for all $\varphi \in \HYperN$.
Since $A_0^{-1}(t)$ is invertible  \eqref{eq:HomogenizedWithP0P1} is equivalent to
\begin{align}\notag
\intOYp \nu A_0\txy \hat{e}_{y,t}(\hat{w}_0\txy) :   \nabla_y\varphi\xy \dyx
\\\notag
+\intOYp A_0^{\top}\txy\nabla_x \hat{q}\tx  \cdot \varphi\xy + \hat{q}_1(t)\div_y(A_0\txy\varphi(x,y)) \dyx
\\\label{eq:HomogenisedSymmetricReference}
= \intOYp (J_0\txy f\tx - A_0\txy\nabla p_b\tx) \cdot \varphi(x,y) \dyx.
\end{align}
Furthermore, \eqref{eq:WeakLimitTransformed2}--\eqref{eq:WeakLimitTransformed3} follow from Lemma \ref{lemma:HomogenisierungDivergence}.
Then, the microscopic incompressibility condition \eqref{eq:WeakLimitTransformed2}, the $Y$-periodicity and the zero boundary values of $\varphi$ on $\Gamma$ imply that 
\begin{align*}
\intYp \nu A_0\txy(\Psi_0^{-\top} \txy \nabla \hat{w}\txy )^\top : \nabla \varphi(y) \dy = 0
\end{align*} for all $\varphi \in H^1_{\Gamma \#}(\Yp)$, which simplifies \eqref{eq:HomogenisedSymmetricReference} to \eqref{eq:WeakLimitTransformed1}. 
Since \eqref{eq:WeakLimitTransformed1}--\eqref{eq:WeakLimitTransformed3} has a unique solution (cf.~Theorem \ref{thm:HomogenisedReferenceExistence}), the convergence holds for the whole sequence. 
\end{proof}

In order to show the existence and uniqueness of the solution of \eqref{eq:WeakLimitTransformed1}--\eqref{eq:WeakLimitTransformed3}, we derive the following inf--sup estimate for the $\div$-conditions.
\begin{lemma}\label{lemma:InfSupLimit}
There exists a constant $C \in \R$ such that 
\begin{align}\label{eq:infsupLimit}
\sup_{v\in \HYperN} \frac{( A_0(t) v, \nabla \phi_0 )_{\Omega \times Y^\mathrm{p}} -  (\div_y(A_0(t) v), \phi_1)_{\Omega \times Y^\mathrm{p}} }{\norm{v}{L^2(\Omega;H^1_{\Gamma\#}(Y^\mathrm{p}))} \norm{(\phi_0, \phi_1)}{H^1_0(\Omega) \times L^2(\Omega; L^2_{0\#}(Y^\mathrm{p}))}} \geq \beta
\end{align}
for a.e.\ $t \in S$ and any $(\phi_0, \phi_1) \in H^1_0(\Omega) \times L^2(\Omega; L^2_{0\#}(Y^\mathrm{p}))$.
\end{lemma}

\begin{proof}
Let $(\phi_0,\phi_1) \in H^1_0(\Omega) \times L^2(\Omega; L^2_0(Y^\mathrm{p}))$.
From the Bogovski\v{i}-operator, we obtain $u \in \HYperN$ such that
\begin{align*}
\div_y(u) = \phi_1, \hspace{1cm}\norm{u}{L^2(\Omega;H^1_{0}(Y^\mathrm{p}))} \leq C \norm{\phi_1}{L^2(\Omega; L^2_{0\#}(Y^\mathrm{p}))}
\end{align*}
for a constant $C$ which only depends on $\Omega$ and $\Yp$ and not on $\phi_1$.
	
We define the functions $v_1, \ldots, v_n \in H_{\Gamma\#}^1(Y^\mathrm{p})^N$ as the solutions of the following Stokes problems:
Find  $(v_i, p_i) \in H_{\Gamma\#}^1(Y^\mathrm{p})^N \times L^2_0(Y^\mathrm{p})$ such that 
\begin{align*}
(\nabla v_i, \nabla \varphi)_{Y^\mathrm{p}} -(p_i, \div (v_i))_{Y^\mathrm{p}} &= (e_i, \varphi)_{Y^\mathrm{p}},
\\
(\div(v_i), \phi)_{Y^\mathrm{p}} &= 0
\end{align*}
for any $(\varphi ,\phi) \in H_{\Gamma\#}^1(Y^\mathrm{p})^N \times L^2_0(\Omega)$.
Choosing $\varphi = v_j$ gives 
\begin{align}\label{eq:DefA}
A \coloneqq
\begin{pmatrix}
\vdots   &   & \vdots
\\
\intYp v_1(y) \dy & \cdots  & \intYp v_n(y) \dy
\\	\vdots   &   & \vdots
\end{pmatrix}
=
\begin{pmatrix}
(\nabla v_1, \nabla v_1)_{Y^\mathrm{p}}   & \cdots & 	(\nabla v_1, \nabla v_n)_{Y^\mathrm{p}}
\\
\vdots  &    & \vdots
\\		(\nabla v_n, \nabla v_1)_{Y^\mathrm{p}}    & \cdots & 	(\nabla v_n, \nabla v_n)_{Y^\mathrm{p}}
\end{pmatrix}.
\end{align}
Since $\{ v_1,  \cdots, v_n \}$ are linearly independent in $H_{\Gamma\#}^1(Y^\mathrm{p})^N$, $A$ is symmetric and positive definite. 
This guarantees that the following boundary-value problem is well-defined:
Find a solution $w \in H^1_0(\Omega)$ such that
\begin{align}\label{eq:DefWvCor}
(A \nabla w , \nabla \varphi )_\Omega = (\nabla \phi_0,  \nabla \varphi)_\Omega + \Big(\intYp u(\cdot,y) \dy, \nabla \varphi\Big)_\Omega
\end{align}
for any $\varphi \in H^1_0(\Omega)$.	
	
By the Theorem of Lax--Milgram, we obtain unique solutions $w \in H^1_0(\Omega)$, which can be estimated by
\begin{align}
\norm{w}{H^1_0(\Omega)} \leq C( \norm{\phi_0}{H^1_0(\Omega)} +  \norm{u}{\Omega \times Y^\mathrm{p}}).
\end{align}
We define
$v\xy\coloneqq  A_0^{-1}\txy   \left(\sum_{i =1}^{n} v_i(y) \ \partial_{x_i} w(x) - u\xy \right)$
and estimate
\begin{align*}
\norm{v}{L^2(\Omega;H^1_{\Gamma\#}(Y^\mathrm{p}))} 
&\leq 
C( \norm{w}{H^1_0(\Omega)} + 	\norm{u}{L^2(\Omega;H^1_{\Gamma\#}(Y^\mathrm{p}))})
\\
&\leq
C (\norm{\phi_0}{H^1_0(\Omega)} + \norm{\phi_1}{L^2(\Omega; L^2_{0\#}(Y^\mathrm{p}))}).
\end{align*}
for $C$ independent of $t$.
Then, we obtain
\begin{align*}
&(A_0(t) v , \nabla \phi_0 )_{\Omega\times Y^\mathrm{p}}
=
\Big(A \nabla w  - \intYp u(\cdot,y) \dy  , \nabla \phi_0 \Big)_{\Omega}
= (\nabla \phi_0, \nabla \phi_0)_\Omega,
\\
&\div_y(A_0(t) v ) 
=
\sum_{i =1}^{n} \div_y( v_i(y) ) \partial_{x_i} w(x) - \div_y(u(x,y)) = -\phi_1(x).
\end{align*}
Using this explicitly constructed $v$, we can conclude \eqref{eq:infsupLimit} for $C>0$, which is independent of $t$.
\end{proof}

\begin{theorem}\label{thm:HomogenisedReferenceExistence}
There exists a unique solution $(\hat{w}_0, \hat{q}, \hat{q}_1)$ $\in L^p(S; \HYperN)\times L^p(S; H^1_0(\Omega))$$\times L^p(S; L^2(\Omega; L^2_{0\#}(Y^\mathrm{p})))$ of \eqref{eq:WeakLimitTransformed1}--\eqref{eq:WeakLimitTransformed3}.
\end{theorem}
\begin{proof}
Note that the existence of a solution for a.e.~$t\in S$ is, up to the measurability with respect to time, already secured by the homogenisation process.
However, it remains to prove the uniqueness.
We rewrite \eqref{eq:WeakLimitTransformed1}--\eqref{eq:WeakLimitTransformed3} in the setting of the generic saddle-point formulation of Lemma \ref{lemma:SaddlePoint}.
Therefore, we define the following time-dependent bilinear forms:
\begin{align}\notag
&a_t :  \HYperN \times \HYperN  \rightarrow \R,
\\
&\phantom{a_t :}(v,w) \mapsto ( \nu A_0(t) \Psi_0^{-\top}(t) \nabla_y v,  \nabla_y w)_{\Omega \times Y^\mathrm{p}},
\\\notag
&b_t : \HYperN \times (H^1_0(\Omega) \times L^2(\Omega; L^2_0(Y^\mathrm{p}))) \rightarrow \R,
\\
&\phantom{a_t :}(v,(p_0,p_1)) \mapsto (A_0^\top(t) \nabla_x p_0,    v)_{\Omega \times Y^\mathrm{p}}
-
(p_1 , \div_y( A_0(t)  v))_{\Omega \times Y^\mathrm{p}}.
\end{align}

Using the time-independent boundedness of the transformation $\psi_0(t)$, the boundedness of $J_0 \geq c_J$ from below and the Poincar\'e inequality for $H^1_{\Gamma\#}(Y^\mathrm{p})$, we obtain
\begin{align} \notag
a_t(v,v) &=( \nu J_0(t) \Psi_0^{-\top}(t) \nabla_y v, \Psi_0^{-\top}(t) \nabla_y v)_{\Omega \times Y^\mathrm{p}}  
\geq \nu c_J \norm{\Psi_0^{-\top}(t) \nabla_y v}{\HYperN}^2
\\\notag
&\geq \nu c_J  \norm{\Psi_0^{\top}}{L^\infty(S \times \Omega \times \Yp)}^{-2} \norm{\nabla_y v}{\Omega\times \Yp}^2 \geq
C \norm{v}{\HYperN}^2,
\\\notag
a_t(v,w)  &\leq \norm{\sqrt{J_0}\Psi_0^{-\top}}{L^\infty(S \times \Omega \times \Yp)}^2
\norm{\nabla_y v}{\Omega\times \Yp}\norm{\nabla_y w}{\Omega\times \Yp}
\\\notag
&\leq 
C \norm{w}{\HYperN} \norm{v}{\HYperN}
\end{align}
for any $v,w \in \HYperN$ for a time-independent constant $C$.

Now, let $(v, (p_0,p_1)) \in \HYperN \times (H^1_0(\Omega) \times L^2(\Omega; L^2_{0\#}(Y^\mathrm{p})))$. Using the product rule, as well as the Poincar\'e inequalities for $H^1_0(\Omega)$ and $H^1_{\Gamma\#}(Y^\mathrm{p})$, yields
\begin{align*}
b_t(v,(p_0,p_1)) = (J_0(t) \Psi^{-\top}_0(t) \nabla_x p_0,    v)_{\Omega \times Y^\mathrm{p}}
-
(p_1 , \div_y( J_0(t) \Psi^{-1}_0(t)  v))_{\Omega \times Y^\mathrm{p}}
\\
\leq C \norm{\nabla p_0}{\Omega} \norm{v}{\Omega \times Y^\mathrm{p}} 
+ C\norm{p_1}{\Omega \times Y^\mathrm{p}}   \norm{v}{\Omega \times Y^\mathrm{p}} + C\norm{p_1}{\Omega \times Y^\mathrm{p}}  \norm{\nabla v}{\Omega \times Y^\mathrm{p}}
\\
\leq C (\norm{p_0}{H^1_0(\Omega)} +  \norm{p_1}{\Omega \times \Yp}) \norm{v}{\HYperN}
\end{align*}
for a time-independent constant $C$.

From Lemma \ref{lemma:InfSupLimit}, we get a time-independent inf--sup constant for $b_t$.

Since the right-hand sides of \eqref{eq:WeakLimitTransformed1}--\eqref{eq:WeakLimitTransformed3} can be bounded pointwise for a.e.~$t \in S$ by $C \in L^{p_s}(S)$, we can conclude for a.e.~$t \in S$ with Lemma \ref{lemma:SaddlePoint} the existence of a unique solution $(\hat{w}(t), \hat{q}(t),\hat{q}_1(t)) \in \HYperN \times H^1_0(\Omega) \times L^2(\Omega; L^2_0(Y^\mathrm{p}))$ such that
\begin{align}
\norm{\hat{w}(t)}{\HYperN} + \norm{\hat{q}(t)}{H^1(\Omega)} + \norm{\hat{q}_1(t)}{L^2(\Omega;L^2_{0 \#}(Y^\mathrm{p}))} \leq C(t)
\end{align}
for $C\in L^p(S)$

Using the same argumentation as in the proof of Theorem \ref{thm:ExistenceEps}, we obtain additionally the measurability of $(\hat{w}_0,\hat{q},\hat{q}_1)$ with respect to time.	
\end{proof}

\section{The limit problem in the evolving domain}\label{section:BackTrafo}
\subsection{Back-transformation of the limit problem}
Now, we transform the two-pressure Stokes problem \eqref{eq:WeakLimitTransformed1}--\eqref{eq:WeakLimitTransformed3} from the cylindrical substitute domain into the actual two-scale domain. The result is the two-pressure Stokes problem \eqref{eq:weakFormV01}--\eqref{eq:weakFormV03}, which does not depend on the transformation $\psi_0$. Thereby, the two-scale-transformation method (cf.~\cite{Wie21}) transforms the two-scale convergence results of the substitute problem (cf.~Theorem \ref{thm:HomogenisationTransformed}) into two-scale convergence results of the untransformed setting (cf.~Theorem \ref{thm:TscLimitUntransformed}).

Moreover, the homogenisation of the Stokes problem yields not only the two-scale convergence for the pressure but a strong convergence for an appropriate extension of it (cf.~Lemma \ref{lemma:ConvergenceQe}). Using the two-scale-transformation method, we can transform the strong convergence of $\hat{Q}$ back and obtain the strong convergence for the back-transformed extension of the pressure $Q' = \hat{Q} \circ \psi_\e^{-1}$ (cf.~Lemma \ref{lemma:ConvergenceQe'}). Indeed $Q'$ is some extension of the pressure of the original problem but this extension is not transformation-independent. Nevertheless, it can be used in order to show the strong convergence for the extension of the pressure $Q_\e$ which we have chosen in Theorem \ref{thm:ConvergenceQe}. This extension is the transformation-independent counterpart of the extension of Lemma \ref{lemma:ConvergenceQe} for the untransformed setting.

In the last step, we separate the $y$-dependency in the two-pressure Stokes problem \eqref{eq:weakFormV01}--\eqref{eq:weakFormV03} and derive the Darcy law for evolving microstructure \eqref{eq:DarcysLawWithVolume1}--\eqref{eq:DarcysLawWithVolume3}.
\begin{theorem}\label{thm:TscLimitUntransformed}
Let $(w_\e,q_\e) \in L^{p_s}(S;H^1_{\Ge(t)}(\Oe(t))) \times L^{p_s}(S;L^2(\Oe(t))$ be the solution of \eqref{eq:WeakFormUntransformed1}-- \eqref{eq:WeakFormUntransformed2} and let $(\widetilde{w_\e}, \widetilde{q}_\e) \in L^{p_s}(S;H^1(\Omega)) \times L^{p_s}(S;L^2(\Omega))$ be the corresponding extensions by $0$.
Then, for a.e.~$t\in S$, $\widetilde{w_\e}(t), \e  \widetilde{\nabla w_\e}(t)$ and $\widetilde{q}_\e(t)$  two-scale converge weakly with respect to the $L^2$-norm to $\widetilde{w_0}(t), \nabla\widetilde{  w_0}(t)$ and $\chi_{\Ypx(t)} q(t)$, respectively, where $\widetilde{w_0}$ is the extension of $w_0$ by $0$ on $S \times \Omega \times Y$ and $(w_0,q,q_1)$ the solution of 
\eqref{eq:weakFormV01}--\eqref{eq:weakFormV03}.
Moreover, for every $p \in (1,2)$, $\widetilde{q}_\e(t)$  two-scale converge weakly with respect to the $L^p$-norm to $\chi_{\Ypx(t)} q(t)$.
\end{theorem}

The transformation-independent two-pressure Stokes problem in the actual two-scale domain is given by:
Find $(w_0, q,q_1) \in L^{p_s}(S;L^2(\Omega;H^1_{\Gamma(t)}(\Ypx(t)))) \times L^{p_s}(S;H^1_0(\Omega)) \times L^{p_s}(S;L^2(\Omega;L^2_0(\Ypx(t))))$ such that
\begin{align}\notag
\intO& \int\limits_{Y^p_x(t)} \nu \nabla_y v_0\txy : \nabla_y \varphi\xy +\nabla_x q\tx \cdot  \varphi\xy \dyx
\\\label{eq:weakFormV01}
-
&\intO \int\limits_{Y^p_x(t)} q_1\txy  \div(\varphi\xy) \dyx =  \intO\int\limits_{Y^p_x(t)}  (f\tx +  \nabla_x p_b\tx)\cdot  \varphi\xy \dyx  
\\\label{eq:weakFormV02}
\intO& \int\limits_{Y^p_x(t)} \div_y (v_0\txy) \  \phi_1 \xy  \dyx = 0
\\\label{eq:weakFormV03}
\intO& \div_x \Big( \int\limits_{Y^\mathrm{p}_x(t)} v_0 \txy \dy \Big) \ \phi_0(x) \dx 
= \intO \int\limits_{Y^p_x(t)} \div_y(v_\Gamma\txy) \dy \ \phi_0(x) \dx
\end{align}
for every $(\varphi, \phi_0,\phi_1) \in L^2(\Omega;H^1_{\Gamma(t) \#}(\Ypx(t))) \times H^1_0(\Omega) \times L^2(\Omega;L^2_0(\Ypx(t)))$.

\begin{proof}
Let $\widetilde{w_\e}, \widetilde{q_\e}$ be defined as in Theorem \ref{thm:TscLimitUntransformed}. Then,  we obtain from
Lemma \ref{lemma:EquivalenceTransformation} that
$\widetilde{w_\e}(t,\psi_\e(t, \cdot_x)) = \widetilde{\hat{w}_\e}(t,\cdot_x)$ and $
\widetilde{q_\e}(t, \psi_\e(t,\cdot_x)) = \hat{Q}_\e(t,\cdot_x)  \chi_\Oe$
for a.e.~$t\in S$.
The two-scale transformation rule (cf.~Theorem \ref{thm:Tsc-Trafo}) provides, for a.e.~$t\in S$, that $\widetilde{w_\e}(t)$ and $\widetilde{ \nabla w_\e}(t)$ two-scale converge to 
$\widetilde{w_0}(t, \cdot_x, \cdot_y) = \hat{w}_0(t,\cdot_x,  \psi_0^{-1}(t, \cdot_x,\cdot_y))$ and $\nabla_y \widetilde{w}_0 \txy$, respectively, where
$\hat{w}_0$ is the two-scale limit of $\widetilde{\hat{w}_\e}$ given by Theorem \ref{thm:HomogenisationTransformed}.
Moreover, the strong convergence of $\hat{Q}_\e(t)$ to $\hat{q}(t)$ in $L^2(\Omega)$ and the strong two-scale convergence of $\chi_\Oe$ to $\chi_\Yp$ with respect to every $L^p$-norm for $p\in (1,\infty)$ gives the strong two-scale convergence of $\widetilde{q_\e}(t) = \chi_\Oe\hat{Q}_\e(t,\cdot_x)  $ to $\chi_\Yp \hat{q}(t)$ with respect to every $L^p$-norm for $p\in (1,2)$. Theorem \ref{thm:Tsc-TrafoStrong} transforms this into the strong two-scale convergence of $\widetilde{q_\e}(t)$ to $\chi_{\Ypx(t)} \hat{q}(t)$  with respect to every $L^p$-norm for $p\in (1,2)$. In order to obtain additionally the weak two-scale convergence with respect to the $L^2$-norm, it is sufficient to show that $\norm{\widetilde{q_\e}(t)}{\Omega}$ is bounded. By transforming $\widetilde{q_\e}$ back and using the uniform boundedness of $J_\e$ and the estimate on $\hat{Q}_\e(t)$ (see \eqref{eq:HatQeBounded}), we obtain this boundedness
\begin{align}
\norm{\widetilde{q_\e}(t)}{\Omega} = 
\norm{\sqrt{J_\e(t)} \hat{Q}_\e}{\Oe} \leq \norm{\sqrt{J_\e(t)} \hat{Q}_\e}{\Omega} \leq C.
\end{align}

We still have to derive the transformation-independent limit problem \eqref{eq:weakFormV01}--\eqref{eq:weakFormV03} in its actual coordinates. Therefore, we test \eqref{eq:WeakLimitTransformed1}--\eqref{eq:WeakLimitTransformed3} with $\hat{\varphi}\xy = \varphi(x, \psi_0\txy)$ and transform the $\Yp$-integral with $\psi_0^{-1}(t,x,\cdot_y)$. Then, we obtain:
\begin{align}
\intO& \int\limits_{Y^p_x(t)} \nu \nabla_y w_0\txy : \nabla_y \varphi\xy \dyx  
\\\notag
&+ \intO \int\limits_{Y^p_x(t)}\Psi_0^\top(t,x,\psi_0^{-1}\txy) (\nabla_x q\tx + \nabla_x p_b\tx) \cdot  \varphi\xy \dyx
\\\label{eq:weakFormW0Back1}
-
&\intO \int\limits_{Y^p_x(t)} \hat{q}_1(t,x, \psi_0^{-1}\txy)  \div(\varphi\xy) \dyx = \intO\int\limits_{Y^p_x(t)} f\tx \cdot  \varphi\xy \dyx  
\\\label{eq:weakFormW0Back2}
\intO& \int\limits_{Y^p_x(t)} \div_y (v_0\txy)   \phi_1 \xy  \dyx = 0
\\\notag
\intO& \div_x\Big( \int\limits_{Y^\mathrm{p}_x(t)}  \Psi_0^{-\top}(t,x,\psi_0^{-1}\txy) v_0 \txy \dy \Big) \phi_0(x) \dx 
\\\label{eq:weakFormW0Back3}
&= \intO \int\limits_{Y^p_x(t)} \div_y(v_\Gamma\txy) \dy \ \phi_0(x) \dx
\end{align}
Note that the transformation coefficients vanish in front of the $y$-derivatives because of the product rule. However, we want to get completely rid of them and  note that
\begin{align*}
\Psi_0^{-\top}(t,x,\psi_0^{-1}\txy) =
\nabla_y \psi_0^{-1}\txy = \1 +\nabla_y \check{\psi}_0^{-1}\txy.
\end{align*}
Thus, we can rewrite the macroscopic pressure terms and obtain after integration by parts
\begin{align}\notag
&\intO \int\limits_{Y^p_x(t)} \Psi_0^{-\top}(t,x,\psi_0^{-1}\txy) (\nabla_x q\tx + \nabla_x p_b\tx) \cdot \varphi\xy \dyx 
\\\notag
=&
\intO \int\limits_{Y^p_x(t)} (\nabla_x q\tx + \nabla_x p_b\tx) \cdot \varphi\xy \dyx 
\\\notag
&+
\intO \int\limits_{Y^p_x(t)} \nabla_y \check{\psi}_0^{-1}\txy (\nabla_x q\tx + \nabla_x p_b\tx) \cdot \varphi\xy \dyx 
\\\notag
=&
\intO \int\limits_{Y^p_x(t)} (\nabla_x q\tx + \nabla_x p_b\tx) \cdot \varphi\xy \dyx
\\\label{eq:RewriteMacroscopicPressures}
&-
\intO \int\limits_{Y^p_x(t)} \check{\psi}_0^{-1}\txy(\nabla_x q\tx + \nabla_x p_b\tx) \div_y(\varphi\xy) \dyx.
\end{align}
The boundary integral, which arises in the integration by parts in \eqref{eq:RewriteMacroscopicPressures}, vanishes on $\partial \Yp \cap \partial Y$ because all the terms are $Y$-periodic and on $\partial \Yp \setminus \partial Y$ because $\varphi =0$ there.
As the last term of \eqref{eq:RewriteMacroscopicPressures} has only a microscopic contribution, we can add it to the microscopic pressure. We define
\begin{align*}
q_1\txy &\coloneqq \hat{q}_1(t,x,\psi_0^{-1}\txy) +  \check{\psi}_0^{-1}\txy \cdot (\nabla_x q\tx + \nabla_x p_b\tx)
\end{align*}
so that the pressure terms of \eqref{eq:weakFormW0Back1} transform to the pressure terms in \eqref{eq:weakFormV01}.
By a similar argumentation, we rewrite the left-hand side of \eqref{eq:weakFormW0Back3}
\begin{align}\notag
\intO \div_x\Big(\int\limits_{Y^p_x(t)} \Psi_0^{-\top}(t,x,\psi_0^{-1}\txy) v_0\txy \dy\Big) \varphi(x) \dx
\\
=
\intO \div_x\Big(\intO\int\limits_{Y^p_x(t)}  v_0\txy + \check{\psi}_0^{-1}\txy  \div_y( v_0\txy) \dy\Big) \varphi(x) \dx
\label{eq:RewriteMacroscopicCompress}.
\end{align}
The second summand on the right-hand side of \eqref{eq:RewriteMacroscopicCompress} vanishes because of the microscopic incompressibility condition \eqref{eq:weakFormW0Back2}. Thus, we have rewritten the left-hand side of \eqref{eq:weakFormW0Back3} into the left-hand side of \eqref{eq:weakFormV03}.
\end{proof}
For the case of a no-slip boundary condition at the interface $\Ge$, in which $v_\Ge\tx = \psi_\e(t,\psi_\e^{-1}\tx)$ models the boundary deformation, we can simplify the right-hand side of the macroscopic compressibility condition \eqref{eq:weakFormV03} in the two-pressure Stokes system.
\begin{cor}\label{cor:vGamma=Deformation}
If $v_\Ge$ is the velocity of the boundary deformation, i.e.~$v_\Ge = \partial_t \psi_\e(t,\psi_\e^{-1}\txy)$, the right-hand side of \eqref{eq:WeakLimitTransformed3}, and equivalently the right-hand side of \eqref{eq:weakFormV03}, can be rewritten into
\begin{align}\label{eq:WeakDirichlet=Deformation}
-\intOYp \div_y(A_0(t) \hat{v}_\Gamma\txy) dy \varphi_0(x) \dx =- \intO \partial_t |\Ypx(t)| \varphi_0(x) \dx
\end{align}	
\end{cor}
\begin{proof} 
First, we note that $v_\Ge\txy = \partial_t \psi_\e(t,\psi_\e^{-1}\txy)$ yields $\hat{v}_\Ge = \partial_t \psi_\e$, which implies $\hat{v}_\Gamma = \partial_t \psi_0$. 
Thus, we can rewrite
\begin{align*}
-\intOYp \div_y(A_0(t) \hat{v}_\Gamma\txy) dy \varphi_0(x) \dx = -\intOYp \div_y(A_0(t) \partial_t \psi_0\txy) dy \varphi_0(x) \dx.
\end{align*}
Then, the Piola identity implies $\div_y(J_0 \Psi_0^{-1} \partial_t \psi_0) = \partial_t J_0$, which gives
\begin{align*}
-\intOYp \div_y(A_0(t) \partial_t \psi_0\txy) dy \varphi_0(x) \dx = -\intOYp \partial_t J_0\txy dy \varphi_0(x) \dx =
\\
-\intO \partial_t \int\limits_{\Ypx(t)} dy \varphi_0(x) \dx =- \intO \partial_t |\Ypx(t)| \varphi_0(x) \dx.
\end{align*}
\end{proof}

In the next step, we consider the limit $\e \to 0$ of the actual fluid velocity $v_\e$. Therefore, we extend $v_\e$ on $\Omega$ by $0$, which is not regularity preserving but conforms with the physical model that no fluid flow happens in the solid phase.
\begin{cor}\label{cor:LimVeps}
Let $v_\e \coloneqq w_\e - v_\Ge \in L^{p_s}(S;H^1_{\Ge(t)}(\Oe(t))$, where $w_\e$ is the solution of \eqref{eq:WeakFormUntransformed1}--\eqref{eq:WeakFormUntransformed2}. Let $\widetilde{v_\e}$ and $\widetilde{\nabla v_\e}$  be the extension by zero on $\Omega \times Y$. Then,
$\widetilde{v_\e}(t)$ and $\e \widetilde{\nabla v_\e}(t)$ two-scale converge to the extension by $0$ of $w_0(t)$ and $\nabla_y w_0(t)$, respectively, where $w_0$ is the solution of \eqref{eq:weakFormV01}--\eqref{eq:weakFormV03}.
\end{cor}
\begin{proof}
We note that $\widetilde{v_\e}(t) - \widetilde{w}_\e(t) = \chi_{\Oe(t)}  v_\Ge$ and  $\widetilde{\nabla v_\e}(t) -  \widetilde{\nabla w}_\e(t) = \chi_{\Oe(t)}  \nabla v_\Ge$.
Since $\norm{v_\Ge(t)}{\Omega} + \e \norm{\nabla v_\Ge(t)}{\Omega} \leq \e C(t)$ for a.e~$t\in S$ for $C \in L^{p_s}(S)$, Theorem \ref{thm:TscLimitUntransformed} gives the desired two-scale convergence.
\end{proof}

By transforming the extension of the pressure back, we obtain the following strong convergence result.
\begin{lemma}\label{lemma:ConvergenceQe'}
Let $Q_\e'\tx \coloneqq \hat{Q}_\e(t,\psi_\e^{-1}\tx)$, where $\hat{Q}_\e$ is given by Lemma \ref{lemma:ConvergenceQe}.
Then, $Q_\e'$ is an extension of $q_\e$, where $q_\e$ is the solution of \eqref{eq:WeakFormUntransformed1}--\eqref{eq:WeakFormUntransformed2}, and
$Q_\e'(t)$ converges strongly to $q(t)$ in $L^p(\Omega)$ for a.e.~$t\in S$ and every $p \in [1,2)$. Moreover, $Q_\e'(t)$ two-scale converges weakly with respect to the $L^2$-norm to $q(t)$.
\end{lemma}
\begin{proof}
Theorem \ref{thm:HomogenisationTransformed} shows that $\hat{Q}_\e(t)$ converges strongly in $L^2(\Omega)$ to $\hat{q}(t)$ for a.e.~$t\in S$. This implies the strong two-scale convergence of $Q_\e(t)$ with respect to the $L^2$-norm. Then, the two-scale transformation method translates this convergence into the strong convergence of $Q_\e'\tx \coloneqq \hat{Q}_\e(t,\psi_\e^{-1}\tx)$ to $q(t)$ with respect to the $L^p$-norm for $p \in(1,2)$ (cf.~Theorem \ref{thm:Tsc-TrafoStrong}). Since the two-scale limit $q(t)$ does not depend on $y$ it does not have to be transformed back. Furthermore, because $q(t)$ is independent of $y$, the strong two-scale convergence with respect to the $L^p$-norm is equivalent to the strong convergence in $L^p(\Omega)$.

In order to prove additionally the weak two-scale convergence with respect to the $L^2$-norm, it is sufficient to show that $\norm{Q_\e'(t)}{\Omega}$ is bounded. This boundedness follows immediately from the uniform boundedness of $J_\e$ and the uniform boundedness of $\norm{\hat{Q}_\e(t)}{\Omega}$ (cf.~\eqref{eq:HatQeBounded}) by $
\norm{Q_\e'(t)}{\Omega} = \norm{\sqrt{J_\e(t)}\hat{Q}_\e(t)}{\Omega} \leq C$.
\end{proof}
Note that the extension $Q_\e'$ of $q_\e$ given by Lemma \ref{lemma:ConvergenceQe'} is not transformation-independent. In particular, if $\psi_\e(k +\e Y) \neq k +\e Y$ for $k \in I_\e$, it can be easily seen that $Q_\e'$ is not constant on $\psi_\e(k +\e \Ys) \neq k + \e Y \cap \Oes(t)$.
However, Lemma \ref{lemma:ConvergenceQe'} makes it possible to prove the strong convergence for the following transformation-independent extension of $q_\e$, which is the counterpart of the extension in the substitute problem $Q_\e$  (see.~\ref{lemma:ConvergenceQe}).
\begin{theorem}\label{thm:ConvergenceQe}
Assume that $|k+ \e Y \cap \Oe(t)| \geq c $ for every $\e>0$ and $k \in I_\e$ with a time- and space-independent constant $c>0$.
Let 
\begin{align}
Q_\e\tx \coloneqq \begin{cases}
q_\e\tx & \textrm{if } x \in \Oe(t),
\\
\frac{1}{| k + \e Y \cap \Oe(t)|} \int\limits_{k + \e Y \cap \Oe(t)} q_\e(t,z) \dz &\textrm{if } x \in k + \e Y  \cap \Oes(t) \textrm{ for } k \in I_\e,
\end{cases}
\end{align}
where $q_\e$ is the second part of the solution of \eqref{eq:WeakFormUntransformed1}--\eqref{eq:WeakFormUntransformed2}.
Then, for a.e.~$t\in S$, $Q_\e(t)$ converges strongly in $L^p(\Omega)$, for every $p\in [1,2)$ to $q(t)$, where $ q\in L^{p_s}(S;H^1(\Omega))$ is the second part of the solution of \eqref{eq:weakFormV01}--\eqref{eq:weakFormV03}.
Moreover, $Q_\e(t)$ two-scale converges weakly with respect to the $L^2$-norm to $q(t)$.
\end{theorem}
\begin{proof}
In order to prove the two-scale convergence, we use the unfolding operator $\Te :L^p(\Omega) \rightarrow L^p(\Omega \times Y)$, which was introduced in \cite{CDG02} and the notations there.  
The unfolding operator allows us to translate between the strong two-scale convergence and the strong convergence in $L^p(\Omega \times Y)$ (cf. \cite{Wie21}).
Thus, $Q_\e(t)$ two-scale converges strongly to $q(t)$ if and only if $\Te(Q_\e(t))$ converges strongly in $L^p(\Omega \times Y)$ to $q(t)$. 
Using the definition of $\Te$, we can rewrite 
\begin{align}\notag
\Te(Q_\e(t))(x,y) = 
\begin{cases}
q_\e(t, \xeY + \e y) &\textrm{if } \xeY + \e y\in \Oe(t),
\\\notag
\frac{1}{|\Oe(t) \cap \xeY + \e Y|} \int\limits_{\xeY + \e Y} \hspace{-0.3cm} \widetilde{q}_\e(t,z) \dz &\textrm{if } \xeY + \e y\in \Oes(t)
\end{cases}
\\\notag
=\Te (\widetilde{q}_\e(t))\xy + \Te(\chi_{\Oes(t)})\xy \frac{1}{|\Oe(t) \cap \xeY+ \e Y|} \intY
\Te( \widetilde{q_\e}(t))(x,z) \dz
\\\label{eq:TeQe}
=
\Te (\widetilde{q}_\e(t))\xy + \Te(\chi_{\Oes(t)})\xy \frac{1}{\intY \Te(\chi_{\Oe(t)})(x,z) \dz} \intY
\Te (\widetilde{q_\e}(t))(x,z) \dz,
\end{align}
where $\widetilde{q}_\e$ is the extension by $0$ of $q_\e$.

In order to pass to the limit $\e \to 0$, we note that $\chi_\Oe$ two-scale converges strongly to $\chi_\Yp$ with respect to the $L^p$-norm for every $p \in (1,\infty)$. Then, Theorem \ref{thm:Tsc-TrafoStrong} implies that $\chi_{\Oe(t)}$ two-scale converges strongly to $\chi_{\Ypx(t)}$ with respect to every $L^p$-norm for  $p \in (1,\infty)$, which is equivalent to 
\begin{align}\label{eq:Conv:Oe(t)}
\Te(\chi_{\Oe(t)}) \to \chi_{\Ypx(t)}  \textrm{ in } L^p(\Omega \times Y) \textrm{ for every }p \in (1,\infty)
\end{align}
and gives also the strong convergence of $
\Te(\chi_{\Oes(t)})$ to $\chi_{Y\setminus\Ypx(t)} $ in $L^p(\Omega \times Y)$ for every $p \in (1,\infty)$.
Using the Cauchy--Schwarz inequality, we obtain additionally
the strong convergence of $\intY \Te(\chi_{\Oe(t)}) (\cdot_x,z) \dz$ to $\intY\chi_{\Ypx(t)}  (\cdot_x,z) \dz$ in $L^p(\Omega)$ for every $p \in (1,\infty)$.
Since $\intY \Te( \chi_{\Oe(t)})(\cdot_x,z) \dz \geq c >0$ is uniformly bounded from below and $\intY \chi_{\Ypx(t)}(\cdot_x,z) \dz = \intYp J_0(t,\cdot_x,z) \dz \geq |\Yp|c_J>0$ as well, we get
\begin{align}\label{eq:Conv:YpOe(t)}
\Big(\intY \Te(\chi_{\Oe(t)}) (\cdot_x,z) \dz \Big)^{-1} \to \Big(\intY\chi_{\Ypx(t)} (z) \dz\Big)^{-1}  \textrm{ in } L^p(\Omega) \textrm{ for every }p \in (1,\infty).
\end{align}
Moreover, we note that $\widetilde{q}_\e(t) = \chi_{\Oe(t)}Q_\e'$ for $Q_\e'$ defined in Lemma \ref{lemma:ConvergenceQe'}.
Then, the strong two-scale convergences of $\chi_{\Oe(t)}$ and of $Q_\e'(t)$, which is given by Lemma \ref{lemma:ConvergenceQe'}, imply the strong two-scale convergence of $\widetilde{q}_\e(t) = \chi_{\Oe(t)}Q_\e'(t)$ to $\chi_{\Ypx(t)}q(t)$ with respect to the $L^p$-norm for every $p \in (1,2)$. Hence, we obtain 
\begin{align}\label{eq:Conv:Teqe}
\Te (\widetilde{q}_\e(t)) \to \chi_{\Ypx(t)}q(t) \textrm{ in } L^p(\Omega \times Y) \textrm{ for every }p \in (1,2).
\end{align} 
Using the Cauchy--Schwarz inequality, we obtain additionally
\begin{align}\label{eq:Conv:YpTeqe}
\intY \Te (\widetilde{q}_\e(t))(\cdot_x,z) \dz \to \intY \chi_{\Ypx(t)}q(t) (\cdot_x,z) \dz \textrm{ in } L^p(\Omega) \textrm{ for every }p \in (1,2).
\end{align}
By combining \eqref{eq:Conv:Oe(t)},\eqref{eq:Conv:YpOe(t)}, \eqref{eq:Conv:Teqe} and \eqref{eq:Conv:YpTeqe}, we can pass to the limit in \eqref{eq:TeQe} and obtain
\begin{align}
\Te(\widetilde{q_\e}(t)) \to \chi_{\Ypx(t)}q(t) + \chi_{Y\setminus \Ypx(t)} \frac{1}{\intY \chi_{\Ypx(t)}(z) \dz } \intY \chi_{\Ypx(t)} q(t) \dz = q(t)
\end{align}
in $L^p(\Omega \times Y)$ for every $p \in (1,2)$.

In order to obtain the weak two-scale convergence with respect to the $L^2$-norm, we have to show the boundedness of $\norm{Q_\e}{\Omega}^2
= \norm{\widetilde{q_\e}(t)}{\Omega(t)}^2 + \norm{Q_\e(t)}{\Omega \setminus \Oe(t)}^2$.
By employing the Cauchy--Schwarz inequality, we obtain 
\begin{align*}
\norm{Q_\e(t)}{\Omega \setminus \Oe(t)}^2
&=
\sum\limits_{k \in I_\e} \int\limits_{k + \e Y} \left(\frac{1}{|k + \e Y \cap \Oe(t)}\int\limits_{k + \e Y} \widetilde{q}_\e(t,z) \dz \right)^2 \dy
\\
&\leq 
\sum\limits_{k \in I_\e} \int\limits_{k + \e Y} C \norm{\widetilde{q_\e}(t)}{k + \e Y}^2 \dy
\leq 
\sum\limits_{k \in I_\e} C \norm{\widetilde{q_\e}(t)}{k + \e Y}^2
\leq C
\norm{q_\e(t)}{\Oe(t)}^2.
\end{align*}
Thus, we estimate with the uniform boundedness of $J_\e$ and the estimate on $\hat{q}_\e$ (cf.~\eqref{eq:EpsIndependentEstimateVP})
\begin{align*}
\norm{Q_\e(t)}{\Omega \setminus \Oe(t)}^2 \leq C\norm{q_\e(t)}{\Oe(t)}^2 = \norm{\sqrt{J_\e(t)}\hat{q}_\e(t)}{\Oe}^2 \leq C \norm{\hat{q}_\e(t)}{\Oe}^2 \leq C.
\end{align*}
\end{proof}

\subsection{The Darcy law for evolving microstructure}
In the last step, we derive the Darcy law \eqref{eq:DarcysLawWithVolume1}--\eqref{eq:DarcysLawWithVolume3} by separating the $y$-dependence in \eqref{eq:weakFormV01}--\eqref{eq:weakFormV03}.
It contains the time- and space-dependent permeability tensor $K \in L^\infty(S\times \Omega)^{N \times N}$, which can be computed explicitly by 
\begin{align}\label{eq:Permeability}
K\tx_{ij} = (\nabla u_i\tx, \nabla u_j\tx )_{Y^\mathrm{p}_x(t)},
\end{align}
where $u_i \in L^\infty(S \times \Omega; H^1_{\Gamma \#}(Y^\mathrm{p}_x(t)))$ are the unique solution 
of the local Stokes problems on the cell domains $Y^\mathrm{p}_x(t)$,
\begin{align}\label{eq:StokesCellProblem1}
&- \Delta_y u_i\txy  - \nabla_y \pi_i \txy = e_i && \textrm{in } Y^\mathrm{p}_x(t),
\\\label{eq:StokesCellProblem2}
&\div ( u_i\txy ) = 0  &&\textrm{in } Y^\mathrm{p}_x(t),
\\\label{eq:StokesCellProblem3}
& u_i\txy = 0 && \textrm{on } \partial \Gamma_x(t),
\\\label{eq:StokesCellProblem4}
&y \mapsto \pi\txy, u_i\txy 	&&\textrm{is $Y$-periodic}.
\end{align}

The corresponding weak formulation of \eqref{eq:DarcysLawWithVolume1}--\eqref{eq:DarcysLawWithVolume3} consists of the following Dirichlet boundary-value problem \eqref{eq:DarcyP} for the pressure and the explicit equation for the fluid velocity \eqref{eq:DarcyV}, where $p = q+p_b$:
Find $q \in L^{p_s}(S;H^1_0(\Omega))$ such that, for a.e.~$t \in S$,
\begin{align}\notag
\intO \frac{1}{\nu}K\tx\nabla_x q\tx \cdot \nabla \varphi \dx = \intO \frac{1}{\nu}  K\tx\nabla_x (f\tx - \nabla p_b\tx) \cdot \nabla \varphi \dx 
\\\label{eq:DarcyP}
-\intOYp \div_y(v_\Gamma\tx) \dy \varphi(x) \dyx
\end{align}
for every $\varphi \in H^1_0(\Omega)$, where $K \in L^\infty(S\times \Omega)^{N \times N}$ is defined by \eqref{eq:Permeability} and let
\begin{align}\label{eq:DarcyV}
v(t) = \frac{1}{\nu} K(t) \left(f(t)- \nabla_x p(t) \right),
\end{align}
where $p = q + p_b$.

In the case of Corollary \ref{cor:vGamma=Deformation}, the last term of \eqref{eq:DarcyP} can be simplified into 
\begin{align*}
-\intOYp \div_y(v_\Gamma\txy) \dy \varphi(x) \dy=\intO \partial_t |\Ypx(t)|  \varphi(x) \dx
\end{align*}

\begin{theorem}\label{thm:Darcy}
Let $(\widetilde{v_\e}, Q_\e)$ be defined by Corollary \ref{cor:LimVeps} and Theorem \ref{thm:ConvergenceQe}, respectively. Then, for a.e.~$t \in S$, $Q_\e(t)$ converges weakly with respect to the $L^2$-norm to $q(t)$ and strongly with respect to the $L^p$-norm for every $p \in (1,2)$, where $q \in L^{p_s}(S;H^1_0(\Omega))$ is the unique solution of 
\eqref{eq:DarcyP}.
Moreover, $\widetilde{v_\e}(t)$ converges weakly with respect to the $L^2$-norm to $v(t)$, where $v \in L^{p_s}(S;L^2(\Omega))$ is given by
\eqref{eq:DarcyV}.
\end{theorem}
\begin{proof}
The linearity of \eqref{eq:weakFormV01} gives
\begin{align}
v_0\txy = \frac{1}{\nu} \sum\limits_{i=1}^N (f_i\tx - \partial_{x_i} (q\tx + p_b\tx)) u_i\txy,
\\
q_1\txy = \frac{1}{\nu} \sum\limits_{i=1}^N ( \partial_{x_i} (q\tx + p_b\tx) - f_i\tx ) \pi_i\txy,
\end{align}
where $(u_i,\pi_i)$ is the solution of \eqref{eq:StokesCellProblem1}-\eqref{eq:StokesCellProblem4} for $i =\{1, \dots, N\}$.
Then, we obtain $v\tx= \frac{1}{\nu} K\tx (f\tx - (\nabla q\tx+ \nabla p_b \tx))$
for
$v\tx \coloneqq\intYp v_0\txy\dy$ by taking the average over $\Yp$, which gives \eqref{eq:DarcyV}.
Moreover, with \eqref{eq:weakFormV03}, we obtain the compressibility condition  $\div(v)= \intYp v_\Gamma\txy \dy$
and in the case of Corollary \ref{cor:vGamma=Deformation}, we can simply it to $\div(v)= \partial_t |\Ypx(t)|$.
Combining this compressibility condition with \eqref{eq:DarcyV} yields \eqref{eq:DarcyP}.
\end{proof}

By stating the compressibility condition $\div(v)= \intYp v_\Gamma\txy \dy= \partial_t |\Ypx(t)|$, which we have derived in the proof of Theorem \ref{thm:Darcy} separately, we obtain the strong formulation of the Darcy law for evolving microstructure \eqref{eq:DarcysLawWithVolume1}--\eqref{eq:DarcysLawWithVolume3}.
It differs in three points from the Darcy law for fixed microstructure \eqref{eq:DarcysLaw}. The first is the time- and space-dependent permeability tensor, which arises from the time- and space-dependent (evolving) microstructure. The second and most interesting difference is the macroscopic compressibility condition, which arises from the homogenisation of the inhomogeneous Dirichlet boundary condition.
The last difference is the Dirichlet boundary condition in \eqref{eq:DarcysLawWithVolume1}--\eqref{eq:DarcysLawWithVolume3} which is caused from the homogenisation of the pressure boundary condition.

\begin{remark}
Instead of the homogenisation for a.e.~$t\in S$ separately, the two-scale convergence with respect to the $L^{p_s}(S;L^p(\Omega))$-norm for $1 < p_s, p<\infty$ could have been used. Thus, the assumption on the data (cf.~Assumption \ref{ass:Data}) can be weakened accordingly.
\end{remark}

\begin{ack}
We would like to thank Markus Gahn for some useful comments on this subject.
\end{ack}

\printbibliography
\end{document}